\documentclass[10pt,reqno]{amsart}

\usepackage{amsmath,amssymb,amsthm}
\usepackage{color,graphicx} 	
\usepackage{enumerate} 		
\usepackage{fullpage}

\numberwithin{equation}{section}

\newtheorem{theorem}{Theorem}[section]
\newtheorem{proposition}[theorem]{Proposition}
\newtheorem{corollary}[theorem]{Corollary}

\newtheorem{lemma}[theorem]{Lemma}

\DeclareMathOperator{\dist}{dist}       
\renewcommand{\div}{\operatorname{div}}
     
\DeclareMathOperator{\loc}{loc}

\DeclareMathOperator{\supp}{supp}
\DeclareMathOperator{\tr}{tr}

\DeclareMathOperator{\pv}{p.v.}

\renewcommand{\O}{\Omega}
\renewcommand{\o}{\omega}

\newcommand{\dd}{\mathrm{d}}
\newcommand{\N}{\mathbb{N}}

\newcommand{\R}{\mathbb{R}} 
\newcommand{\Rn}{\mathbb{R}^n} 
\newcommand{\Rnn}{\mathbb{R}^{n \times n}}
\newcommand{\Rnnsd}{\mathbb{R}^{n \times n}_{\mathrm{s,d}}}

\newcommand{\ssubset}{\subset \! \subset}
\newcommand{\Sn}{\mathbb{S}^{n-1}}
\renewcommand{\vec}{\mathbf}    

\newcommand{\weakc}{\rightharpoonup}

\newcommand{\G}{\Gamma}
\newcommand{\mc}{\mathcal}
\newcommand{\mf}{\mathfrak}
\newcommand{\p}{\partial}

\renewcommand{\d}{\delta}
\newcommand{\g}{\gamma}
\newcommand{\e}{\varepsilon}
\newcommand{\s}{\sigma}
\renewcommand{\a}{\alpha}
\renewcommand{\b}{\beta}
\renewcommand{\l}{\lambda}

\newcommand{\f}{\varphi}
\renewcommand{\r}{\rho}

\newcommand{\Hn}{\mathcal{H}^{n-1}}

\newcommand{\ove}{\overline}
\newcommand{\haz}{\widehat}
\newcommand{\into}{\int_\O}

\newcommand{\xx}{\vec x' {-} \vec x}

\def\Xint#1{\mathchoice
{\XXint\displaystyle\textstyle{#1}}%
{\XXint\textstyle\scriptstyle{#1}}%
{\XXint\scriptstyle\scriptscriptstyle{#1}}%
{\XXint\scriptscriptstyle\scriptscriptstyle{#1}}%
\!\int}
\def\XXint#1#2#3{{\setbox0=\hbox{$#1{#2#3}{\int}$}
\vcenter{\hbox{$#2#3$}}\kern-.5\wd0}}

\def\dashint{\Xint-}

\title{ Quasistatic elastoplasticity via Peridynamics: \\ existence
  and localization}

\author{Martin Kru\v{z}\'ik}
\address[Martin Kru\v{z}\'ik]{Institute of Information Theory and Automation,
Academy of Sciences,
Pod vod\' arenskou ve\v z\' \i\ 4, 18208, Prague 8, Czech Republic.}
\email{kruzik@utia.cas.cz}
\urladdr{http://staff.utia.cas.cz/kruzik/}

\author{Carlos Mora-Corral}
\address[Carlos Mora-Corral]{Universidad Aut\' onoma de Madrid, Departamento de Matem\' aticas,
Facultad de Ciencias, C/ Tom\'as y Valiente 7, 
28049 Madrid, Spain.}
\email{carlos.mora@uam.es}
\urladdr{http://www.uam.es/carlos.mora}

\author{Ulisse Stefanelli}
\address[Ulisse Stefanelli]{Faculty of Mathematics, University of Vienna, 
Oskar-Morgenstern-Platz 1, 1090 Wien, Austria  and  Istituto di Matematica
Applicata e Tecnologie Informatiche \textit{{E. Magenes}}, v. Ferrata 1, 27100
Pavia, Italy.}
\email{ulisse.stefanelli@univie.ac.at}
\urladdr{http://www.mat.univie.ac.at/$\sim$stefanelli}

\begin{document}

\begin{abstract}
 Peridynamics is a nonlocal continuum-mechanical theory based on
minimal regularity on the deformations. Its key trait is that of
replacing local constitutive relations featuring spacial differential
operators with integrals over differences of displacement fields over
a suitable positive interaction range. The
advantage of such perspective is that of directly including nonregular situations, in
which discontinuities in the displacement
field may occur. In the linearized elastic setting, the  mechanical
foundation of the theory and its mathematical amenability have been
thoroughly analyzed in the last years.

We present here the extension of Peridynamics to linearized
elastoplasticity. This calls for considering the time  evolution of elastic
and plastic variables, as the effect of a combination of elastic
energy storage and plastic energy dissipation mechanisms. The
quasistatic evolution problem is variationally reformulated and solved
by time discretization. In addition, by a rigorous evolutive
$\Gamma$-convergence argument we prove that the nonlocal
peridynamic model converges to classic local elastoplasticity as the
interaction range goes to zero.
\end{abstract}

\maketitle

\section{ Introduction}

Peridynamics is a nonlocal mechanical theory based on the
formulation of equilibrium systems in integral terms
instead of differential relations. Forces acting on a
material point are obtained as a combined effect of interactions with other
points in a neighborhood. This results in an integral featuring a radial
weight which modulates the influence of nearby points in terms of
their distance \cite{Emmrich}. 

Introduced by {\sc Silling} \cite{Silling00}, and
extended in \cite{Silling-Lehoucq08,Silling10},  Peridynamics
 is particularly suited to model situations
where displacements tend to develop discontinuities, such as in the case of
cracks or dislocations \cite{Handbook,Emmrich15}. In addition,
 this nonlocal formulation   is capable of
integrating discrete and continuous descriptions, possibly serving as
a connection between
multiple scales \cite{Seleson09}. As such, it is particularly
appealing in order to model the ever smaller scales of modern
technological applications \cite{survey}.  

In the frame of the peridynamic theory, the elastic equilibrium problem for a linear
homogeneous isotropic body subject to the external force of density $\vec b(\vec x)\in \Rn$  can be
variationally formulated as the minimization of the purely elastic
energy 
\begin{align*}
 E_{\r} (\vec u) = & \b \int_{\O} \mf{D}_{\r} (\vec u) (\vec x)^2 \, \dd \vec x + \a \int_{\O} \int_{\O} \r (\vec x' {-} \vec x) \left( \mc{D} (\vec u) (\vec x, \vec x') - \frac{1}{n} \mf{D}_{\r} (\vec u) (\vec x) \right)^2 \dd \vec x' \, \dd \vec x - \int_{\O} \vec b (\vec x) \cdot \vec u (\vec x) \, \dd \vec x
\end{align*}
among displacements $\vec u(\vec x)\in \Rn$ from a reference configuration $\O\subset \Rn$, subject to
boundary conditions. Here $\r:\Rn\to[0,\infty)$ is an integral kernel
modeling the strength of interactions with respect to the distance of the
points $\vec x '$ and $\vec x$, the term $\mc{D}(\vec u)  (\vec x, \vec x') $
plays the role of a nonlocal elastic strain, projected in the direction $(\xx)/|\xx|$, namely
\begin{equation}
 \mc{D} (\vec u) (\vec x, \vec x') = \frac{\left( \vec u (\vec x') - \vec u (\vec x) \right) \cdot (\vec x' {-} \vec x)}{|\vec x' {-} \vec x|^2},\label{D}
\end{equation}
and $\mf{D}_{\r} (\vec u)(\vec x)$ is a nonlocal analogue of the
divergence and is given by
\begin{equation}
 \mf{D}_{\r} (\vec u) (\vec x) = \pv \int_{\O} \r (\vec x' {-} \vec x) \mc{D}(\vec u) (\vec x, \vec x') \, \dd \vec x' , \qquad \text{for a.e. }\vec x \in \O ,\label{DD}
\end{equation}
where $\pv$ stands for the principal value. The
positive material parameters $\a$ and $\b$ are related to the shear and
bulk moduli of the material, respectively. 

The purely elastic energy
$E_\r$ has been recently intensively investigated \cite{DuGuLeZh12,MeDu14,Silling10,Silling-Lehoucq08}. In particular, by suitably qualifying assumptions on
the kernel $\r$, the force $\vec b$, and by imposing boundary
conditions (see below) one can check that $E_\r$ admits a unique
minimizer $\vec u_\r$. In addition, in \cite{MeDu15} it is proved that, in the limit of
vanishing interaction range, that is for $\rho $ converging to a Dirac
delta function centered at $\vec 0$, the nonlocal solutions $\vec u_\r$
converge to the  unique  solution of the classical local elastic equilibrium
system,  namely the minimizer of  
\begin{align*}
 E_0 (\vec u)  
=\frac{\lambda}{2} \int_{\O} \div \vec u (\vec x) ^2 \, \dd \vec x +
\mu \int_{\O}|\nabla^s \vec u (\vec x)|^2 \, \dd \vec x - \int_{\O} \vec b (\vec x) \cdot \vec u (\vec x) \, \dd \vec x.
\end{align*} The symbol $\nabla^s$ stands for the linearized strain
$\nabla^s\vec u = (\nabla \vec u + (\nabla \vec u)^\top)/2$ and the Lam\'e coefficients $\lambda$ and $\mu$ are related to $\alpha$, $\beta$, and $n$
via \cite[App.\ A]{MeDu15}
\begin{equation}
 \lambda = 2\beta -  \frac{4\alpha}{n(n+2)}, \quad
\mu= \frac{2\alpha}{n+2}.\label{eq:lame}
\end{equation}
Note that $\mu>0$ and $n\lambda + 2 \mu>0$, making the elastic energy
coercive. Indeed, calling $\vec u_\r$ and $\vec u_0$ the minimizers of
$E_\r$ and $E_0$, respectively,  the convergence of $\vec u_\r$ to
$\vec u_0$  follows from the $\Gamma$-convergence of $E_{\r}$ to $E_0$ \cite{DalMaso,DeGiorgi}.

The focus of this paper is on extending the elastic
theory to encompass plastic effects as well. This calls for
considering the
plastic strain $\vec P \in \Rnnsd$ (symmetric and deviatoric tensors)
as an additional variable and to define the elastoplastic energy as
  \begin{align}
 F_{\r} (\vec u, \vec P) &=  \b \int_{\O} \mf{D}_{\r} (\vec u) (\vec x)^2 \, \dd \vec x + \a \int_{\O} \int_{\O} \r (\vec x' {-} \vec x) \left( \mc{E} (\vec u, \vec P) (\vec x, \vec x') - \frac{1}{n} \mf{E}_{\r} (\vec u, \vec P) (\vec x) \right)^2 \dd \vec x' \, \dd \vec x \nonumber\\
 & + \g \int_{\O} |\vec P (\vec x)|^2 \, \dd \vec x- \int_{\O} \vec b (\vec x) \cdot \vec u (\vec x) \, \dd \vec x \label{F}
\end{align}
where the nonlocal elastic strain, projected in direction $(\xx)/|\xx|$,
features now  the additional contribution of the plastic strain as
\begin{equation}
 \mc{E} (\vec u, \vec P) (\vec x, \vec x') = \frac{\left( \vec u
     (\vec x') - \vec u (\vec x) - \vec P (\vec x) (\vec x' {-} \vec
     x)\right) \cdot (\vec x' {-} \vec x)}{|\vec x' {-} \vec x|^2}.\label{E}
\end{equation}
Correspondingly, we define
\begin{equation}
 \mf{E}_{\r} (\vec u, \vec P) (\vec x) = \pv \int_{\O} \r (\vec x' {-} \vec x) \mc{E}(\vec u, \vec P) (\vec x, \vec x') \, \dd \vec x',\label{EE}
\end{equation}
which again plays the role of a nonlocal divergence of $\vec
u$. Indeed, although it depends on $\vec P$, one can check that such
dependence vanishes when the kernel $\r$ tends to the Dirac delta
function at $\vec 0$ as $\vec P $ is assumed to be deviatoric, 
see Lemma \ref{le:Eddiv}.a.  

With
respect to the purely elastic case of $E_\r$, an additional
$\gamma$-term is here considered. This models kinematic hardening and $\g>0$
is the corresponding hardening coefficient. Note that the whole energy
$F_\r$ is quadratic in $(\vec u, \vec P)$. This results in a
linearized theory of elastoplasticity, although of a nonlocal
nature. The corresponding localized elastoplastic energy is the
classical
\begin{align*}
 F_0 (\vec u, \vec P) =\frac{\lambda}{2} \int_{\O} \div \vec u (\vec x) ^2 \, \dd \vec x +
\mu \int_{\O}|\nabla^s \vec u (\vec x) - \vec P (\vec x)|^2 \, \dd \vec x - \int_{\O} \vec b (\vec x) \cdot \vec u (\vec x) \, \dd \vec x +
 \g \int_{\O} |\vec P (\vec x)|^2 \, \dd \vec x .
\end{align*}

Elastoplastic evolution requires the specification of the plastic
dissipation mechanism. We follow here the classical von Mises choice: given some
yield stress $ \sigma_y >0$, we specify the energy dissipated in order to pass
from the plastic state $\vec P_0$ to $\vec P_1$ as 
$$H(\vec P_1 {-} \vec P_0) =  \sigma_y \int_\O |\vec P_1(\vec x) {-} \vec
P_0(\vec x)|\, \dd \vec x.$$
We let the action of the external force density $\vec
b$ to be depending on time and correspondingly investigate
trajectories $t \mapsto (\vec u_\r(t),\vec P_\r(t))$ solving the
quasistatic evolution system
\begin{align}
  \partial_{\vec u} F_\r (\vec u_\r(t), \vec P_\r(t),t) &= \vec 0  
   , \label{var100}\\
\partial_{\dot{\vec P}} H(\dot{\vec P}_\r(t)) + \partial_{\vec P} F_\r
  (\vec u_\r(t), \vec P_\r(t),t) &\ni \vec 0.  \label{var200}
\end{align} 
The symbol $\partial$ above is the subdifferential in the sense of
convex analysis and the dot in \eqref{var200} denotes the time derivative. Relation \eqref{var100} corresponds to the weak
formulation of the
quasistatic equilibrium system. Relation
\eqref{var200} is the plastic flow rule instead. In particular, as $H$
is not smooth in $\vec 0$, relation \eqref{var200} is actually a pointwise
inclusion. Quasistatic evolution in the present nonlocal peridynamic
elastoplastic context is then driven by the pair of functionals
$(F_\r,H)$ whereas the choice $(F_0,H)$ correspond to classical
localized elastoplasticity.

The two main results of this paper are the following:

\begin{center}
  \begin{minipage}{1.0\linewidth}
    \begin{itemize}
    \item {\bf (Theorem \ref{thm1})} Under suitable assumptions on the
      data and the kernel $\r$, there exists a unique trajectory
      $t \mapsto (\vec u_\r(t),\vec P_\r(t))$ solving the nonlocal
      quasistatic evolution system.\\[-1mm]
    \item {\bf (Theorem \ref{thm:conv_quasi})} If $\r$ converges to
      the Dirac delta  function  at $\vec 0$, then the solutions
      $t \mapsto (\vec u_\r(t),\vec P_\r(t))$ converge to the unique
      quasistatic evolution $t \mapsto (\vec u_0(t),\vec P_0(t))$ for
      local classical elastoplasticity.
    \end{itemize}
  \end{minipage}
\end{center}

 In the hyperelastic case, some  corresponding variational theory and its
rigorous relation to local elasticity has recently been settled
in \cite{BeMo14,Bellido15,MeDu15}.
To our knowledge, this paper contributes the first variational peridynamic model
including internal variables. Note that damage and plastic effects in the frame
of Peridynamics have already been considered in \cite{Emmrich16,Hu12}
and \cite{Madenci16}, respectively.
The analysis of the well-posedness of the
quasistatic evolution (Theorem \ref{thm1}) and the localization proof (Theorem \ref{thm:conv_quasi}) seem unprecedented
out of the elastic context. 

The well-posedness result is based on time discretization.
After explaining the functional setup (Section \ref{se:functional}), in Section \ref{se:static} we investigate incremental problems of the form 
$$\min\big(F_\r(\vec u,\vec P,t_i) + H(\vec P {-} \vec P_{\rm
  old})\big)$$
where the previous plastic state $\vec P_{\rm old} $ and the time $t_i$
are given. These minimization problems are proved to be well-posed
(Subsection \ref{se:ex_static})  and to
converge in the sense of $\Gamma$-convergence to the corresponding
localized counterparts as the kernel $\r$ approaches the Dirac delta
function at $\vec 0$ (Subsection \ref{se:Glimit}). By passing to the
limit in the time
discretized problem as the time step goes to zero, one recovers the
unique solution to the quasistatic evolution system (Section
\ref{se:quasistatic}). Such limit passage is made possible by the quadratic nature of the
energy (Subsection \ref{se:well_posedness_quasistatic}).

The localization result is derived by applying the general theory of
evolutive $\Gamma$-convergence for rate-independent evolution from \cite{MRS}. In
particular, such possibility rests upon the $\Gamma$-convergence of
the energies and
the specification of a recovery sequence for a suitable combination of
energy and dissipation terms  (Subsection \ref{se:local_limit_quasi}). This again crucially exploits the fact that
energies are quadratic.

\section{ Functional setup}\label{se:functional}

 We devote this section to present our assumptions and
introduce some notation. In the following, we will use lower-case bold
letters for vectors in $\Rn$ and capitalized bold letters for tensors
in $\Rnn$. In particular $\vec a \cdot \vec b$ is the standard scalar product. We use the symbol $\vec I$ for the identity, $\vec A: \vec B= \tr
(\vec A^\top\vec B)$ for the standard contraction product, $|\vec A|^2 = \vec
A: \vec A$ for the norm, and recall that an infinitesimal rigid
displacement is a function of the form $\vec x \mapsto \vec S \vec x +
\vec v$ with $\vec S \in \Rnn$ skew-symmetric and $\vec v \in \Rn$. 

Let $\Omega \subset \Rn$ (open, bounded and Lipschitz) be the reference
configuration of the body. The state of the medium is described by the
pair $(\vec u,\vec P)$, where  $\vec u: \Omega \times (0,T) \to \Rn$
is the displacement and $\vec P: \Omega \times (0,T) \to \Rnnsd$ is
the plastic strain. Here, $T>0$
is a final reference time and $\Rnnsd$ stands for the set of symmetric
trace-free (deviatoric)  matrices, namely $\tr \vec P(\vec x,t) =
0$. We also use the symbol  $\mc{R}$ for the $L^2 (\O, \Rn)$ subset
of infinitesimal rigid displacements in $\O$. We
will indicate by $\| \cdot \|_p$ the norm of any $L^p$ space on $\Omega$.

Let an integral kernel $\r \in L^1(\Rn,[0,\infty))$ with
$\|\r\|_1=n$ be given. We define for all $(\vec u, \vec P)\in
L^2(\O,\Rn)\times L^2(\O,\Rnnsd)$ the
quantities $\mc{D}(\vec u)(\vec x,\vec x')$, $\mc{E}(\vec u, \vec
P)(\vec x,\vec x')$, $\mf{D}_\r(\vec u)(\vec x)$, and $\mf{E}_\r(\vec u,
\vec P)(\vec x)$ from \eqref{D}--\eqref{DD} and \eqref{E}--\eqref{EE} for a.e.\ $\vec x$ and $\vec x'$ in $\O$. We can hence
define the elastoplastic energy $F_\r$ in \eqref{F} on the whole of
$L^2(\O,\Rn)\times L^2(\O,\Rnnsd)$, possibly taking the value $\infty$.

Note that, by Jensen's (or H\"older's) inequality, 
\begin{align}\label{eq:DleqS}
 \mf{D}_{\r} (\vec u) (\vec x)^2 &\leq \int_{\O} \r (\vec x' {-} \vec
 x) \, \dd \vec x' \, \pv \int_{\O} \r (\vec x' {-} \vec x) \mc{D}
 (\vec u) (\vec x, \vec x')^2 \, \dd \vec x' \nonumber\\
&\leq n \, \pv \int_{\O} \r (\vec x' {-} \vec x) \mc{D} (\vec u)
(\vec x, \vec x')^2 \, \dd \vec x' \qquad \text{for a.e. } \vec x \in \O .
\end{align}
In particular, we have that $\mf{D}_{\r} (\vec u) \in L^2(\O)$  if  
\[
 \int_{\O} \int_{\O} \r (\vec x' {-} \vec x) \mc{D} (\vec u) (\vec x, \vec x')^2 \, \dd \vec x' \, \dd \vec x < \infty .
\]
Accordingly, we define 
\begin{equation*}
 |\vec u|_{\mc{S}_{\r}} = \left( \int_{\O} \int_{\O} \r (\vec x' {-} \vec x) \mc{D} (\vec u) (\vec x, \vec x')^2 \, \dd \vec x' \, \dd \vec x \right)^{1/2} , \qquad \|\vec u\|_{\mc{S}_{\r}} = \left( \|\vec u\|_2^2 + |\vec u|_{\mc{S}_{\r}}^2 \right)^{1/2} .
\end{equation*}
and the space
\begin{equation*}
 \mc{S}_{\r} (\O) = \left\{ \vec u \in L^2 (\O, \Rn) : |\vec u|_{\mc{S}_{\r}} < \infty \right\} .
\end{equation*}
It is immediate to see that $\left| \cdot \right|_{\mc{S}_{\r}}$ is a seminorm and $\left\| \cdot \right\|_{\mc{S}_{\r}}$ is a norm in $\mc{S}_{\r} (\O)$.
In fact, $\mc{S}_{\r} (\O)$ is a separable Hilbert space, as shown in \cite[Th.\ 2.1]{MeDu15}.
One can easily see that $|\vec u|_{\mc{S}_{\r}} = 0$ if and only if $\vec u \in \mc{R}$.

In the following, we will impose homogeneous Dirichlet boundary
conditions  on $\vec u$ by asking
the displacement $\vec u$ to belong to  the closed
subspace $V$ of $L^2 (\O, \Rn)$ given by
\[
 V = \left\{ \vec u \in L^2 (\O, \Rn) :  \vec u = \vec 0 \text{
     a.e.\ in } \o \right\} 
\]
where $\o \subset \O$ is a measurable subset with non-empty interior
such that $\Omega \setminus \overline \o$ is Lipschitz. 
With this choice, it is proved in \cite{DuGuLeZh13ELAS} that $V\cap \mc{R}=\{\vec 0\}$ so that
(nonnull) infinitesimal rigid-body motions are ruled out; see also
\cite{DuGuLeZh13,GuLe10}.

Although we stick with this choice of $V$ in the following, let us
mention that other boundary conditions can be considered as
well. Nonhomogeneous Dirichlet conditions can be easily dealt with and
we
refer to \cite{DuGuLeZh13ELAS} for some detail concerning Neumann
conditions.

As in \eqref{eq:DleqS}, by Jensen's (or H\"older's) inequality, 
\begin{align}\label{eq:ErleqE}
 \mf{E}_{\r} (\vec u, \vec P) (\vec x)^2 & \leq n \, \pv \int_{\O} \r (\vec
                                           x' {-} \vec x) \mc{E} (\vec
                                           u, \vec P) (\vec x, \vec
                                           x')^2 \, \dd \vec x'
                                           \qquad \text{for a.e. } \vec x \in \O .
\end{align}
In addition, since
\begin{equation}\label{eq:DEp}
 \mc{D} (\vec u) (\vec x, \vec x') = \mc{E} (\vec u, \vec P) (\vec x, \vec x') + \frac{\vec P (\vec x) (\vec x' {-} \vec x) \cdot (\vec x' {-} \vec x)}{|\vec x' {-} \vec x|^2} , \qquad \text{a.e. } \vec x, \vec x' \in \O ,
\end{equation}
we also have the bounds
\[
 \mc{D} (\vec u) (\vec x, \vec x')^2 \leq 2 \left( \mc{E} (\vec u, \vec P) (\vec x, \vec x')^2 + |\vec P (\vec x)|^2 \right)
\]
and, hence,
\begin{equation}\label{eq:DleqEp}
 \int_{\O} \int_{\O} \r (\vec x' {-} \vec x) \mc{D} (\vec u) (\vec x, \vec x')^2 \, \dd \vec x' \, \dd \vec x \leq 2 \int_{\O} \int_{\O} \r (\vec x' {-} \vec x) \mc{E} (\vec u, \vec P) (\vec x, \vec x')^2 \, \dd \vec x' \, \dd \vec x + 2n \int_{\O} \left| \vec P (\vec x) \right|^2 \, \dd \vec x .
\end{equation}


In view of \eqref{eq:DleqS}, \eqref{eq:ErleqE},
and \eqref{eq:DleqEp}, we have that  the elastoplastic energy
$F_{\r} (\vec u, \vec P)$ is finite in $(\vec u, \vec P) \in L^2(\O,\Rn)\times L^2(\O,\Rnnsd)$ if and only if   
\[
 \int_{\O} \int_{\O} \r (\vec x' {-} \vec x) \mc{E} (\vec u, \vec P) (\vec x, \vec x')^2 \, \dd \vec x' \, \dd \vec x < \infty .
\]
Accordingly, we define
\[
  |(\vec u, \vec P)|_{\mc{T}_{\r}} = \left( \int_{\O} \int_{\O} \r (\vec x' {-} \vec x) \mc{E} (\vec u, \vec P) (\vec x, \vec x')^2 \, \dd \vec x' \, \dd \vec x \right)^{1/2} , \qquad 
 \|(\vec u, \vec P)\|_{\mc{T}_{\r}} = \left( \|\vec u\|_2^2 + \|\vec P\|_2^2 + |(\vec u, \vec P)|_{\mc{T}_{\r}}^2 \right)^{1/2}
\]
and the space
\[
 \mc{T}_{\r} (\O) = \left\{ (\vec u, \vec P) \in L^2 (\O, \Rn) \times L^2 (\O, \Rnnsd) : |(\vec u, \vec P)|_{\mc{T}_{\r}} < \infty \right\} .
\]
It is easy to see that $\left| \cdot \right|_{\mc{T}_{\r}}$ is a seminorm and $\left\| \cdot \right\|_{\mc{T}_{\r}}$ is a norm in $\mc{T}_{\r} (\O)$.
We have, in fact, the following result.

\begin{lemma}\label{le:Tr}
We have that $\mc{T}_{\r} (\O) = \mc{S}_{\r} (\O) \times L^2 (\O, \Rnnsd)$, and the norm $\left\| \cdot \right\|_{\mc{T}_{\r}}$ is equivalent to the product norm in $\mc{S}_{\r} (\O) \times L^2 (\O, \Rnnsd)$.
In addition, $\mc{T}_{\r} (\O)$ is a separable Hilbert space.
\end{lemma}
\begin{proof}
By the triangle inequality,
\[
 |(\vec u, \vec P)|_{\mc{T}_{\r}} \leq |(\vec u, \vec 0)|_{\mc{T}_{\r}} + |(\vec 0, \vec P)|_{\mc{T}_{\r}} = |\vec u|_{\mc{S}_{\r}} + |(\vec 0, \vec P)|_{\mc{T}_{\r}} \quad \text{and} \quad |\vec u|_{\mc{S}_{\r}} = |(\vec u, \vec 0)|_{\mc{T}_{\r}} \leq |(\vec u, \vec P)|_{\mc{T}_{\r}} + |(\vec 0, \vec P)|_{\mc{T}_{\r}} .
\]
Now, $| \mc{E} (\vec 0, \vec P) (\vec x, \vec x') | \leq |\vec P (\vec x)|$ for a.e.\ $\vec x, \vec x' \in \O$, so $|(\vec 0, \vec P)|_{\mc{T}_{\r}}^2 \leq n \| \vec P \|_2^2$.
This shows the equivalence of norms.

Finally, $\mc{T}_{\r} (\O)$ is a separable Hilbert space because so is $\mc{S}_{\r} (\O)$ (see \cite[Th.\ 2.1]{MeDu15}).
\end{proof}

For future reference, recall that the proof of Lemma \ref{le:Tr} has shown that
\begin{equation}\label{eq:uSupT}
 |\vec u|_{\mc{S}_{\r}} \leq |(\vec u, \vec P)|_{\mc{T}_{\r}} + \sqrt{n} \| \vec P \|_2 \quad \text{and} \quad |(\vec u, \vec P)|_{\mc{T}_{\r}} \leq |\vec u|_{\mc{S}_{\r}} + \sqrt{n} \| \vec P \|_2 .
\end{equation}

A crucial tool in the following is the nonlocal Korn inequality, which we take from \cite[Prop.\ 2.7]{MeDu15}.

\begin{proposition}[Nonlocal Korn inequality]\label{pr:Korn}
 There  exists $C>0$ such that $\| \vec u \|_2^2 \leq C | \vec u |_{S_{\r}}^2$ for all $\vec u \in V$.
\end{proposition}

The following result is proved in \cite[Lemma 2.1]{Mengesha12} (see also \cite[Eq.\ (15)]{MeDu15}).

\begin{lemma}\label{le:SrW12}
There exists $C>0$ such that for all $\vec u \in H^1 (\O, \Rn)$,
\[
 \left| \vec u  \right|_{\mc{S}_{\r}}^2 \leq C \, n \left\| \nabla^s \vec u \right\|_2^2 .
\]
\end{lemma}
We remark that the constant $C$ in Lemma \ref{le:SrW12} does not
depend on $\r$.

%

\section{ Incremental problem}\label{se:static}

Let us now turn our attention to the incremental
elastoplastic problem. Given the
plastic strain $\vec P_{\rm old}\in L^2(\O, \Rnnsd)$, it consists in finding $$(\vec u,
\vec P) \in Q=V \times L^2(\O, \Rnnsd)$$ that minimizes the
incremental functional 
\begin{equation}
   F_{\r} (\vec u, \vec P) + H(\vec P - \vec P_{\rm old}). \label{eq:incremental}
\end{equation}

In this section we prove the  well-posedness of the incremental problem
(Subsection
\ref{se:ex_static}) as well as the convergence of its solutions the
solution of its local counterpart as $\delta \to 0$ (Subsection
\ref{se:Glimit}).

 In order to possibly apply the Direct Method to the incremental problem
\eqref{eq:incremental}, the coercivity of $F_{\r}$ will be
instrumental. We check it in the following.

\begin{lemma}[Coercivity of the energy]\label{le:Fcoerc}
There exists $c>0$ such that 
for all $(\vec u, \vec P) \in  Q$, 
\[
 F_{\r} (\vec u, \vec P) \geq  c  \|(\vec u, \vec P)\|_{\mc{T}_{\r}}^2 -\frac{1}{ c} .
\]
\end{lemma}
\begin{proof}
 Assume with no loss of generality that  $(\vec u, \vec P) \in \mc{T}_{\r}$.
For any $0 < \eta < 1$ we have 
\begin{equation}\label{eq:Fgeq1}
 \left( \mc{E} (\vec u, \vec P) (\vec x, \vec x') - \frac{1}{n} \mf{E}_{\r} (\vec u, \vec P) (\vec x) \right)^2 \geq (1-\eta) \mc{E} (\vec u, \vec P) (\vec x, \vec x')^2 - (\eta^{-1} - 1) \frac{1}{n^2} \mf{E}_{\r} (\vec u, \vec P) (\vec x)^2 ,
\end{equation}
for a.e.\ $\vec x, \vec x' \in \O$.
On the other hand, thanks to \eqref{eq:DEp} we have
\[
 \mf{E}_{\r} (\vec u, \vec P) (\vec x) = \mf{D}_{\r} (\vec u) (\vec x) - \pv \int_{\O} \r (\vec x' {-} \vec x) \frac{\vec P (\vec x) (\vec x' {-} \vec x) \cdot (\vec x' {-} \vec x)}{|\vec x' {-} \vec x|^2} \dd \vec x' .
\]
In fact,
\[
 \pv \int_{\O} \r (\vec x' {-} \vec x) \frac{\vec P (\vec x) (\vec x' {-} \vec x) \cdot (\vec x' {-} \vec x)}{|\vec x' {-} \vec x|^2} \dd \vec x' = \int_{\O} \r (\vec x' {-} \vec x) \frac{\vec P (\vec x) (\vec x' {-} \vec x) \cdot (\vec x' {-} \vec x)}{|\vec x' {-} \vec x|^2} \dd \vec x'
\]
since
\[
 \left| \int_{\O} \r (\vec x' {-} \vec x) \frac{\vec P (\vec x) (\vec x' {-} \vec x) \cdot (\vec x' {-} \vec x)}{|\vec x' {-} \vec x|^2} \dd \vec x' \right| \leq \int_{\O} \r (\vec x' {-} \vec x) \dd \vec x' |\vec P (\vec x)| \leq n |\vec P (\vec x)| .
\]
Therefore,
\[
 \left| \mf{E}_{\r} (\vec u, \vec P) (\vec x) \right| \leq \left| \mf{D}_{\r} (\vec u) (\vec x) \right| + n \left| \vec P (\vec x) \right| .
\]
Consequently,
\begin{equation*}
 \mf{E}_{\r} (\vec u, \vec P) (\vec x)^2 \leq 2 \mf{D}_{\r} (\vec u) (\vec x)^2 + 2n^2 \left| \vec P (\vec x) \right|^2 , \qquad  \| \mf{E}_{\r} (\vec u, \vec P) \|_2^2 \leq 2 \| \mf{D}_{\r} (\vec u) \|_2^2 + 2n^2 \|\vec P\|_2^2
\end{equation*}
and
\begin{equation}\label{eq:Fgeq2}
 \int_{\O} \int_{\O} \r (\vec x' {-} \vec x) \mf{E}_{\r} (\vec u, \vec P) (\vec x)^2 \, \dd \vec x \leq n \left\| \mf{E}_{\r} (\vec u, \vec P) \right\|_2^2 \leq 2 n \| \mf{D}_{\r} (\vec u) \|_2^2 + 2n^3 \|\vec P\|_2^2 .
\end{equation}
Therefore, by \eqref{eq:Fgeq1} and \eqref{eq:Fgeq2} we have
\begin{equation}\label{eq:Fgeq3}
\begin{split}
 & \int_{\O} \int_{\O} \r (\vec x' {-} \vec x) \left( \mc{E} (\vec u, \vec P) (\vec x, \vec x') - \frac{1}{n} \mf{E}_{\r} (\vec u, \vec P) (\vec x) \right)^2 \dd \vec x' \dd \vec x \\
 \geq & (1-\eta) |(\vec u, \vec P)|_{\mc{T}_{\r}}^2 - (\eta^{-1} -1) \frac{2}{n} \left( \| \mf{D}_{\r} (\vec u) \|_2^2 + n^2 \|\vec P\|_2^2\right) .
\end{split}
\end{equation}
On the other hand, for any $\eta_1 >0$ we have that
\begin{equation}\label{eq:Fgeq4}
 \left| \int_{\O} \vec b \cdot \vec u \, \dd \vec x \right| \leq \|\vec b\|_2 \|\vec u\|_2 \leq \frac{\|\vec b\|_2^2}{2 \eta_1} + \frac{\eta_1 \|\vec u\|_2^2}{2} .
\end{equation}
Using \eqref{eq:Fgeq3} and \eqref{eq:Fgeq4}, we find that
\begin{align*}
 F_{\r} (\vec u, \vec P) \geq & \left[ \b - \frac{2}{n} \a (\eta^{-1} -1) \right] \| \mf{D}_{\r} (\vec u) \|_2^2 +  \a (1 - \eta) |(\vec u, \vec P)|_{\mc{T}_{\r}}^2  \\
 & + \left[ \g - 2n (\eta^{-1}-1) \right]  \|\vec P\|_2^2 - \frac{\eta_1}{2} \|\vec u\|_2^2 - \frac{1}{2\eta_1} \| \vec b \|_2^2 .
\end{align*}
Choosing $0 < \eta < 1$ such that
\[
 \b - \frac{2}{n} \a (\eta^{-1} -1) \geq 0 \quad \text{and} \quad \g - 2n (\eta^{-1}-1) > 0 ,
\]
we have that inequality 

\begin{equation}\label{eq:Fgeq}
  F_{\r} (\vec u, \vec P) \geq c \left( |(\vec u, \vec P)|_{\mc{T}_{\r}}^2 + \|\vec P\|_2^2 \right) - \frac{\eta_1}{2} \|\vec u\|_2^2 -
  \frac{1}{2\eta_1} \| \vec b \|_2^2 
 \end{equation}
  is proved for some $c>0$. 
By Proposition \ref{pr:Korn} and estimate \eqref{eq:uSupT}, we have
\[
  \| \vec u \|_2^2 \leq C |\vec u|_{\mc{S}_{\r}}^2 \leq 2 C \left( |(\vec u, \vec P)|_{\mc{T}_{\r}}^2 + n \| \vec P \|_2^2 \right) \leq 2 n C \left( |(\vec u, \vec P)|_{\mc{T}_{\r}}^2 + \| \vec P \|_2^2 \right) ,
\]
so
\begin{equation}\label{eq:Fgeq5}
 \frac{c}{2} \left( |(\vec u, \vec P)|_{\mc{T}_{\r}}^2 + \|\vec P\|_2^2 \right) + \frac{c}{2} \left( |(\vec u, \vec P)|_{\mc{T}_{\r}}^2 + \|\vec P\|_2^2 \right) \geq \frac{c}{2} \left( |(\vec u, \vec P)|_{\mc{T}_{\r}}^2 + \|\vec P\|_2^2 \right) + \frac{c}{4 n C} \|\vec u\|_2^2 .
\end{equation}
Using \eqref{eq:Fgeq} and \eqref{eq:Fgeq5} we obtain
\[
  F_{\r} (\vec u, \vec P) \geq \frac{c}{2} \left( |(\vec u, \vec P)|_{\mc{T}_{\r}}^2 + \|\vec P\|_2^2 \right) + \frac{c}{4 n C} \|\vec u\|_2^2 - \frac{\eta_1}{2} \|\vec u\|_2^2 - \frac{1}{2\eta_1} \| \vec b \|_2^2 .
\]
Choosing $\eta_1 >0$ so that
\[
 \frac{c}{4 n C} - \frac{\eta_1}{2} > 0
\]
we  prove the  estimate of the statement.
\end{proof}

 The semicontinuity of the second term of $F_{\r}$ will ensue
from the following control on the projected stress. 

\begin{lemma}[Projected-stress control]\label{le:linearbounded}
The transformation $T_{\r}$ that assigns each $(\vec u, \vec P)$ to the map
\[
 (\vec x, \vec x') \mapsto \r (\vec x' {-} \vec x)^{\frac{1}{2}} \left[ \mc{E} (\vec u, \vec P) (\vec x, \vec x') - \frac{1}{n} \mf{E}_{\r} (\vec u, \vec P) (\vec x) \right]
\]
is linear and bounded from $\mc{T}_{\r} (\O)$ to $L^2 (\O \times \O)$.
Moreover, there exists $C>0$, not depending on $\r$, such that for all $(\vec u, \vec P) \in \mc{T}_{\r} (\O)$,
\[
 \left\| T_{\r} (\vec u, \vec P) \right\|_{L^2 (\O \times \O)} \leq C \left| (\vec u, \vec P) \right|_{\mc{T}_{\r} (\O)} .
\]
\end{lemma}
\begin{proof}
The operators $\mc{E}$ and $\mf{E}_{\r}$ are clearly linear, and, hence, so is $T_{\r}$.
The operator
\[
  (\vec x, \vec x') \mapsto \r (\vec x' {-} \vec x)^{\frac{1}{2}} \mc{E} (\vec u, \vec P) (\vec x, \vec x')
\]
is bounded simply because
\[
 \int_{\O} \int_{\O} \r (\vec x' {-} \vec x) \mc{E} (\vec u, \vec P) (\vec x, \vec x')^2 \, \dd \vec x' \, \dd \vec x = |(\vec u, \vec P)|_{\mc{T}_{\r}}^2 .
\]
Analogously, the operator
\[
  (\vec x, \vec x') \mapsto \r (\vec x' {-} \vec x)^{\frac{1}{2}} \mf{E}_{\r} (\vec u, \vec P) (\vec x)
\]
is bounded because, thanks to \eqref{eq:ErleqE},
\[
 \int_{\O} \int_{\O} \r (\vec x' {-} \vec x) \mf{E}_{\r} (\vec u, \vec P) (\vec x)^2 \, \dd \vec x' \, \dd \vec x \leq n \int_{\O} \mf{E}_{\r} (\vec u, \vec P) (\vec x)^2 \, \dd \vec x \leq n^2 |(\vec u, \vec P)|_{\mc{T}_{\r}}^2  .
\]
This concludes the proof.
\end{proof}

\subsection{Well-posedness of the incremental problem}\label{se:ex_static}
 
A key feature of the energy functional $F_{\rho}$ is its strict convexity, which delivers the existence and uniqueness of minimizers.

\begin{proposition}[Strict convexity of $F_{\r}$]\label{prop:strict}
The functional $F_{\rho}$ is strictly convex in $(V \cap \mc{S}_{\r}) \times L^2 (\O, \Rnnsd)$.
\end{proposition}
\begin{proof}
The operators $\mf{D}_{\r}$ and $T_{\r}$ (see Lemma \ref{le:linearbounded}) are linear, which readily implies that
$F_{\rho}$ is convex.
Let $(\vec u_1, \vec P_1) , (\vec u_2, \vec P_2) \in (V \cap \mc{S}_{\r}) \times L^2 (\O, \Rnnsd)$ and $\l \in (0,1)$ satisfy
\[
 F_{\r} (\l (\vec u_1, \vec P_1) + (1-\l) (\vec u_2, \vec P_2)) = \l F_{\r} (\vec u_1, \vec P_1) + (1-\l) F_{\r} (\vec u_2, \vec P_2) .
\]
Since the norms in $L^2 (\O, \Rnnsd)$ and in $L^2 (\O \times \O)$
are strictly convex, we find that $\vec P_1 = \vec P_2$ a.e.\ and $T_{\r} (\vec u_1, \vec P_1) = T_{\r} (\vec u_2, \vec P_2)$ a.e. 
Calling $\vec v = \vec u_1 - \vec u_2$, we infer that $\vec v \in V$ and $T_{\r} (\vec v, \vec 0) = 0$.
Thus, $|\vec v|_{\mc{S}_{\r}} = 0$, so, by Proposition \ref{pr:Korn}, $\vec v = \vec 0$ and, hence, $\vec u_1 = \vec u_2$ a.e.
\end{proof}

\begin{theorem}[Well-posedness of the incremental problem]\label{th:existence}
Let  $\vec P_{\rm old} \in L^2 (\O, \Rnnsd) $ be given.
Then there exists a unique minimizer of  $ (\vec u, \vec P)
\mapsto F_{\r} (\vec u, \vec P) + H(\vec P {-} \vec P_{\rm old})$ in $Q$. 
\end{theorem}
\begin{proof}
Call $G_{\r} : Q \to \R\cup \{\infty\}$ the function $G_{\r} (\vec u, \vec P) = F_{\r} (\vec u, \vec P) + H(\vec P {-} \vec P_{\rm old})$.
By Lemma \ref{le:Tr}, it is enough to show existence and uniqueness of minimizers of $G_{\r}$ in $(V \cap \mc{S}_{\r}) \times L^2 (\O, \Rnnsd)$ (recall that $F_\r=\infty$ if $\vec u \notin \mc{S}_{\r}$).
By Lemma \ref{le:Fcoerc}, $G_{\r}$
is bounded from below, so it admits a minimizing sequence $\{ (\vec
u_j, \vec P_j) \}_{j \in \N}$ in $(V \cap \mc{S}_{\r})
\times L^2 (\O, \Rnnsd)$.
By Lemma \ref{le:Fcoerc} again, $\{ (\vec u_j, \vec P_j) \}_{j \in \N}$ is bounded in $\mc{T}_{\r}$.
By Lemma \ref{le:Tr}, $\{ \vec u_j \}_{j \in \N}$ is bounded in
$\mc{S}_{\r}$ and $\{ \vec P_j \}_{j \in \N}$ is bounded in 
$L^2(\O;\Rnnsd)$. 
As $V$ is a closed subspace of $L^2 (\O, \Rn)$, it is also a closed subspace of $\mc{S}_{\r}$.
Therefore, there exists $(\vec u_0, \vec P_0) \in (V \cap \mc{S}_{\r}) \times L^2 (\O, \Rnnsd)$ such that, for a subsequence (not relabelled), $\vec u_j \weakc  \vec u_0$ in $\mc{S}_{\r}$ and $\vec P_j \weakc \vec P_0$ in $L^2 (\O, \Rnnsd)$ as $j \to \infty$.

Bound \eqref{eq:ErleqE} tells us that $\mf{E}_{\r}$ is a linear bounded operator from $\mc{T}_{\r}$ to $L^2 (\O)$.
Having in mind that $\mc{D} (\vec u)= \mc{E} (\vec u, \vec 0)$ and $\mf{D}_{\r} (\vec u)= \mf{E}_{\r} (\vec u, \vec 0)$, we obtain that the operator $\mf{D}_{\r} : \mc{S}_{\r} \to L^2 (\O)$ is linear and bounded.
By Lemma \ref{le:linearbounded}, the map $T_{\r}$ defined therein is linear and bounded.
Altogether, $G_{\r}$ is the sum of continuous functions with respect to the strong topology of $\mc{S}_{\r} \times L^2 (\O, \Rnnsd)$.
On the other hand, thanks to Proposition \ref{prop:strict}, $G_{\r}$ is strictly convex as a sum of the strictly convex function $F_{\r}$ and the convex function $(\vec u, \vec P) \mapsto H(\vec P {-} \vec P_{\rm old})$.
Consequently, $G_{\r}$ is lower semicontinuous with respect to the weak topology of $\mc{S}_{\r} \times L^2 (\O, \Rnnsd)$.
Thus,
\[
 G_{\r} (\vec u_0, \vec P_0) \leq \liminf_{j \to \infty} G_{\r} (\vec u_j, \vec P_j)
\]
and, hence, $(\vec u_0, \vec P_0)$ is a minimizer of $G_{\r}$.
The uniqueness of minimizers is an immediate consequence of the strict convexity of $G_{\r}$.
\end{proof}

\subsection{ Localization  limit}\label{se:Glimit}

 We shall now check that, as $\r$ tends to the Dirac delta
function at
$\vec 0$, the unique solution $(\vec u_{\d},\vec P_{\d})$ of the nonlocal
incremental problem \eqref{eq:incremental} converges to the unique solution
of the incremental problem for local classical linearized
elastoplasticity. To this aim, let us specify that the local elastoplastic
energy $F_0 : Q \to \R \cup \{ \infty\}$ is given by
\begin{align*}
 F_0 (\vec u, \vec P) = & \b \int_{\O} \div \vec u (\vec x) ^2 \, \dd \vec x + \a \, n \int_{\O} \dashint_{\Sn} \left( (\nabla \vec u (\vec x) - \vec P (\vec x)) \vec z \cdot \vec z - \frac{1}{n} \div \vec u (\vec x) \right)^2 \, \dd \Hn (\vec z) \, \dd \vec x \\
 & - \int_{\O} \vec b (\vec x) \cdot \vec u (\vec x) \, \dd \vec x +
 \g \int_{\O} |\vec P (\vec x)|^2 \, \dd \vec x \nonumber\\
& =\frac{\lambda}{2} \int_{\O} \div \vec u (\vec x) ^2 \, \dd \vec x +
\mu \int_{\O}|\nabla^s \vec u (\vec x) - \vec P (\vec x)|^2 \, \dd \vec x - \int_{\O} \vec b (\vec x) \cdot \vec u (\vec x) \, \dd \vec x +
 \g \int_{\O} |\vec P (\vec x)|^2 \, \dd \vec x 
\end{align*}
for $\vec u \in H^1(\Omega,\Rn)$, and $F_0 (\vec u, \vec P) = \infty$ otherwise.
The numbers $\l, \mu$ are given by \eqref{eq:lame}.
Correspondingly, the local incremental elastoplastic problem reads as
follows: Given the
previous plastic strain $\vec P_{\rm old}\in L^2(\O, \Rnnsd)$ find $(\vec u,
\vec P) \in Q$  minimizing 
\begin{equation}
   F_0 (\vec u, \vec P) + H(\vec P{-} \vec P_{\rm old}). \label{eq:incremental_local}
\end{equation}

The proof of existence and uniqueness of the minimizer $(\vec u, \vec P) \in (V \cap H^1(\Omega,\Rn)) \times  L^2(\O, \Rnnsd)$ is standard.

We
start by computing  the $\G$-limit of the functional $F_{\d}$ as
$\r$ tends to the Dirac delta function at $\vec 0$ \cite{DalMaso,DeGiorgi}.
The precise assumptions of the family of kernels $\{ \r_{\d}
\}_{\d>0} \subset L^1(\Rn,[0,\infty))$  with $  \|
  \r_\d\|_1=n$ are as follows:
each $\r_{\d}$ is radial, i.e., there exists $\bar{\r}_{\d} : [0, \infty) \to [0, \infty)$ such that $\r_{\d} (\vec x) = \bar{\r}_{\d} (|\vec x|)$; moreover,
\begin{align}
  & \text{the map } [0,\infty) \ni r \mapsto r^{-2} \bar{\r}_{\d}
  (r)  \text{ is decreasing}, 
  \label{newR}\\
  & \text{and} \ 
 \displaystyle \lim_{\d \to 0} \int_{\Rn \setminus B(\vec 0, r)}
   \r_{\d} (\vec x) \, \dd \vec x = 0 \  \text{ for all } r >0 . 
  \label{newR2}
\end{align}

This set of assumptions (or a slight variant of it) is typical in the analysis of the convergence from a nonlocal functional to a local one; see \cite{BoBrMi01,Brezis02,Ponce04,Ponce04b,MeDu15}.
For ease of notation,  in the following  the subscript  ${\r}$  used in the
previous sections in $F_{\r}$, $\mf{D}_{\r}$, $\mf{E}_{\r}$, $T_{\r}$
and so on is replaced by the subscript  ${\d}$,  meaning that the kernel involved is $\r_{\d}$.

In this section we prove the $\G$-convergence of $F_{\d}$ to $F_0$ as $\d \to 0$ in $L^2 (\O, \Rn) \times L^2 (\O, \Rnnsd)$ endowed with the strong topology in $L^2 (\O, \Rn)$ and the weak topology in $L^2 (\O, \Rnnsd)$, or, equivalently, in $H^1 (\O, \Rn) \times L^2 (\O, \Rnnsd)$ endowed with the weak topology.

First we show that $\mf{E}_{\d} (\vec u, \vec P)$ is an approximation of $\div \vec u$.

\begin{lemma}[Convergence of the divergence]\label{le:Eddiv}
Let $\vec u \in H^1 (\O, \Rn)$ and $\vec P \in L^2 (\O, \Rnnsd)$.
The following holds:
\begin{enumerate}[a)]
\item\label{item:Eddiva} $\mf{E}_{\d} (\vec u, \vec P) \to \div \vec u$ as $\d \to 0$ in $L^2 (\O)$.

\item\label{item:Eddivb} For each $\d>0$ let $\vec u_{\d} \in L^2 (\O, \Rn)$ and $\vec P_{\d} \in L^2 (\O, \Rnnsd)$.
Assume $\vec u_{\d} \to \vec u$ in $L^2 (\O, \Rn)$ and $\vec P_{\d} \weakc \vec P$ in $L^2 (\O, \Rnnsd)$ as $\d \to 0$.
Suppose further that $\sup_{\d>0} |\vec u_{\d}|_{\mc{S}_{\d}} < \infty$. 
Then $\mf{E}_{\d} (\vec u_{\d}, \vec P_{\d}) \weakc \div \vec u$ as $\d \to 0$ in $L^2 (\O)$.

\end{enumerate}
\end{lemma}
\begin{proof}
We start with \emph{\ref{item:Eddiva})}.
For each $\d>0$ we define the operator $\mf{P}_{\d} : L^2 (\O, \Rnnsd) \to L^2 (\O)$ by
\begin{equation*}
 \mf{P}_{\d} (\vec P) (\vec x) = \int_{\O} \r_{\d} (\vec x'- \vec x) \frac{\vec P (\vec x) (\vec x'- \vec x) \cdot (\vec x'- \vec x)}{|\vec x'- \vec x|^2} \, \dd \vec x' , \qquad \text{a.e. } \vec x \in \O . 
\end{equation*}
Clearly, we have
\begin{equation}\label{eq:EdDdPd}
 \mf{E}_{\d} (\vec u, \vec P) = \mf{D}_{\d} (\vec u) - \mf{P}_{\d} (\vec P) .
\end{equation}
It was proved in \cite[Lemma 3.1]{MeDu15} that $\mf{D}_{\d} (\vec u) \to \div \vec u$ in $L^2 (\O)$ as $\d \to 0$. 
We shall show that $\mf{P}_{\d} (\vec P) \to 0$ in $L^2 (\O)$.
We can express, for a.e.\ $\vec x \in \O$,
\begin{equation}\label{eq:Pdchange}
 \mf{P}_{\d} (\vec P) (\vec x) = \int_{\O - \vec x} \r_{\d} (\tilde{\vec x}) \frac{\vec P (\vec x) \tilde{\vec x} \cdot \tilde{\vec x}}{|\tilde{\vec x}|^2} \, \dd \tilde{\vec x} ,
\end{equation}
so
\begin{equation}\label{eq:Pleqp}
 \left| \mf{P}_{\d} (\vec P) (\vec x) \right| \leq n \left| \vec P (\vec x) \right| .
\end{equation}
Now let $A \ssubset \O$ and let $0 < r < \dist (A, \p \O)$.
Note that $B(\vec 0, r) \subset \O - \vec x$ for any $\vec x \in A$.
By \eqref{eq:Pdchange} and Lemma \ref{le:formulasB}, we have, for a.e.\ $\vec x \in A$,
\[
 \mf{P}_{\d} (\vec P) (\vec x) = \int_{(\O - \vec x) \setminus B (\vec 0, r)} \r_{\d} (\tilde{\vec x}) \frac{\vec P (\vec x) \tilde{\vec x} \cdot \tilde{\vec x}}{|\tilde{\vec x}|^2} \, \dd \tilde{\vec x} ,
\]
so
\[
 \left| \mf{P}_{\d} (\vec P) (\vec x) \right| \leq \int_{\Rn \setminus B (\vec 0, r)} \r_{\d} (\tilde{\vec x}) \, \dd \tilde{\vec x} \left| \vec P (\vec x) \right|
\]
and, consequently,
\begin{equation}\label{eq:intAPd}
 \int_A \mf{P}_{\d} (\vec P) (\vec x)^2 \, \dd \vec x \leq \left( \int_{\Rn \setminus B (\vec 0, r)} \r_{\d} (\tilde{\vec x}) \, \dd \tilde{\vec x}\right)^2 \left\| \vec P \right\|_2^2 .
\end{equation}
Thanks to  \eqref{newR2}, we obtain that $\mf{P}_{\d} (\vec P) \to 0$ in $L^2 (A)$ as $\d \to 0$.
Now, bound \eqref{eq:Pleqp} implies that the family $\{ \mf{P}_{\d} (\vec P)^2\}_{\d > 0}$ is equiintegrable, so in fact $\mf{P}_{\d} (\vec P) \to 0$ in $L^2 (\O)$ as $\d \to 0$.

Now we show \emph{\ref{item:Eddivb})}.
In \cite[Lemma 3.6]{MeDu15} it was proved that $\mf{D}_{\d} (\vec u_{\d}) \weakc \div \vec u$ in $L^2 (\O)$ as $\d \to 0$.
Thanks to \eqref{eq:EdDdPd}, it remains to show that $\mf{P}_{\d} (\vec P_{\d}) \weakc 0$ in $L^2 (\O)$, and for this we will show that $\{ \mf{P}_{\d} (\vec P_{\d})\}_{\d>0}$ is bounded in $L^2 (\O)$ and that $\mf{P}_{\d} (\vec P_{\d}) \to 0$ in $L^2_{\loc} (\O)$.

Let $\d>0$.
Thanks to \eqref{eq:Pleqp} we have $\left| \mf{P}_{\d} (\vec P_{\d}) \right| \leq n \left| \vec P_{\d} \right|$, so $\{ \mf{P}_{\d} (\vec P_{\d})\}_{\d>0}$ is bounded in $L^2 (\O)$.
Now let $A \ssubset \O$ and let $0 < r < \dist (A, \p \O)$.
By \eqref{eq:intAPd} we have that
\[
 \int_A \mf{P}_{\d} (\vec P_{\d}) (\vec x)^2 \, \dd \vec x \leq \left( \int_{\Rn \setminus B (\vec 0, r)} \r_{\d} (\tilde{\vec x}) \, \dd \tilde{\vec x}\right)^2 \left\| \vec P_{\d} \right\|_2^2 .
\]
Using \eqref{newR2} and the fact that $\{ \vec P_{\d} \}_{\d>0}$ is bounded in $L^2 (\O, \Rnn)$, we conclude that $\mf{P}_{\d} (\vec P_{\d}) \to 0$ in $L^2 (A)$ as $\d \to 0$, which finishes the proof.
\end{proof}

As a preparation for the $\G$-limit $F_{\d} \to F$ as $\d \to 0$, we start with the pointwise limit.

\begin{proposition}[Pointwise convergence of $F_{\d}$]\label{prop:limitFdup}
Let $\vec u \in H^1 (\O, \Rn)$ and $\vec P \in L^2 (\O, \Rnnsd)$.
Then
\[
 \lim_{\d \to 0} F_{\d} (\vec u, \vec P) = F_0 (\vec u, \vec P) .
\]
\end{proposition}
\begin{proof}
Obviously, we only have to show that
\begin{equation}\label{eq:limDddiv}
 \lim_{\d \to 0} \int_{\O} \mf{D}_{\d} (\vec u) (\vec x)^2 \, \dd \vec x = \int_{\O} \div \vec u (\vec x) ^2 \, \dd \vec x
\end{equation}
and
\begin{equation}\label{eq:limitE-Ed}
\begin{split}
 & \lim_{\d \to 0} \int_{\O} \int_{\O} \r_\d (\vec x' {-} \vec x) \left( \mc{E} (\vec u, \vec P) (\vec x, \vec x') - \frac{1}{n} \mf{E}_{\d} (\vec u, \vec P) (\vec x) \right)^2 \dd \vec x' \, \dd \vec x \\
 = & n \int_{\O} \dashint_{\Sn} \left( (\nabla \vec u (\vec x) - \vec P (\vec x)) \vec z \cdot \vec z - \frac{1}{n} \div \vec u (\vec x) \right)^2 \, \dd \Hn (\vec z) \, \dd \vec x .
\end{split}
\end{equation}
As mentioned in Lemma \ref{le:Eddiv}, the limit $\mf{D}_{\d} (\vec u) \to \div \vec u$ in $L^2 (\O)$ as $ \d \to 0$ was shown in \cite[Lemma 3.1]{MeDu15}, so we have equality \eqref{eq:limDddiv}.

We divide the proof of \eqref{eq:limitE-Ed} in two steps, according to the regularity of $\vec u$ and $\vec P$.

\smallskip

\emph{Step 1.} We assume additionally that $\vec u \in C^1 (\bar{\O}, \Rn)$ and $\vec P \in C (\bar{\O}, \Rnnsd)$.

Since $\vec u \in C^1 (\bar{\O}, \Rn)$, there exists an increasing bounded function $\s : [0,\infty) \to [0, \infty)$ with
\begin{equation}\label{eq:lims0}
 \lim_{t \to 0} \s (t) = 0
\end{equation}
such that for all $\vec x, \vec x' \in \O$,
\[
 \left| \nabla \vec u (\vec x') - \nabla \vec u (\vec x) \right| \leq \s (|\vec x' {-} \vec x|) .
\]
As $\O$ is a Lipschitz domain, a standard result shows that there exists $c\geq 1$ such that for all $\vec x, \vec x' \in \O$, we have
\begin{equation}\label{eq:differenceq}
 \left| \vec u (\vec x') - \vec u (\vec x) \right| \leq c \left\| \nabla \vec u \right\|_{\infty} \left| \vec x' {-} \vec x \right|
\end{equation}
and
\[
 \left| \vec u (\vec x') - \vec u (\vec x) - \nabla \vec u (\vec x) (\vec x' {-} \vec x) \right| \leq |\vec x' {-} \vec x| c \, \s (|\vec x' {-} \vec x|) .
\]
For simplicity of notation, we relabel $c \, \s$ as $\s$ and, hence, assume that for all $\vec x, \vec x' \in \O$,
\begin{equation}\label{eq:diffgrad}
 \left| \vec u (\vec x') - \vec u (\vec x) - \nabla \vec u (\vec x) (\vec x' {-} \vec x) \right| \leq |\vec x' {-} \vec x| \s (|\vec x' {-} \vec x|) .
\end{equation}
Note that \eqref{eq:differenceq} implies that
\begin{equation}\label{eq:boundEgrad}
 \left| \mc{E} (\vec u, \vec P) (\vec x, \vec x') \right| \leq c \left\| \nabla \vec u \right\|_{\infty} + \left\| \vec P \right\|_{\infty} .
\end{equation}

Now we show that
\begin{equation}\label{eq:Ed-div}
 \lim_{\d \to 0} \int_{\O} \int_{\O} \r_{\d} (\vec x' {-} \vec x) \left[ \left( \mc{E} (\vec u, \vec P) (\vec x, \vec x') - \frac{1}{n} \mf{E}_{\d} (\vec u, \vec P) (\vec x) \right)^2 - \left( \mc{E} (\vec u, \vec P) (\vec x, \vec x') - \frac{1}{n} \div \vec u (\vec x) \right)^2 \right] \dd \vec x' \, \dd \vec x = 0 .
\end{equation}
We have
\begin{align*}
 & \left| \int_{\O} \int_{\O} \r_{\d} (\vec x' {-} \vec x) \left[ \left( \mc{E} (\vec u, \vec P) (\vec x, \vec x') - \frac{1}{n} \mf{E}_{\d} (\vec u, \vec P) (\vec x) \right)^2 - \left( \mc{E} (\vec u, \vec P) (\vec x, \vec x') - \frac{1}{n} \div \vec u (\vec x) \right)^2 \right] \dd \vec x' \, \dd \vec x \right| \\
 = & \frac{1}{n^2} \left|  \int_{\O} \int_{\O} \r_{\d} (\vec x' {-} \vec x) \left( \div \vec u (\vec x) - \mf{E}_{\d} (\vec u, \vec P) (\vec x) \right) \left( 2 n^2 \mc{E} (\vec u, \vec P) (\vec x, \vec x') - \mf{E}_{\d} (\vec u, \vec P) (\vec x) - \div \vec u (\vec x) \right) \dd \vec x' \, \dd \vec x \right| \\
 \leq & \frac{1}{n^2} \left( \int_{\O} \int_{\O} \r_{\d} (\vec x' {-} \vec x) \left( \div \vec u (\vec x) - \mf{E}_{\d} (\vec u, \vec P) (\vec x) \right)^2 \dd \vec x' \, \dd \vec x \right)^{\frac{1}{2}} \\
 & \quad \times \left( \int_{\O} \int_{\O} \r_{\d} (\vec x' {-} \vec x) \left( 2 n^2 \mc{E} (\vec u, \vec P) (\vec x, \vec x') - \mf{E}_{\d} (\vec u, \vec P) (\vec x) - \div \vec u (\vec x) \right)^2 \dd \vec x' \, \dd \vec x \right)^{\frac{1}{2}} .
\end{align*}
Thanks to \eqref{eq:ErleqE} and \eqref{eq:boundEgrad}, the second term of the right-hand side is bounded by a constant times
\[
 \left\| \nabla \vec u \right\|_{\infty} + \left\| \vec P \right\|_{\infty} ,
\]
while the first term tends to zero as $\d \to 0$ thanks to Lemma \ref{le:Eddiv}.
Thus, limit \eqref{eq:Ed-div} is proved.

Now we show
\begin{equation}\label{eq:divudivu}
\begin{split}
 & \lim_{\d \to 0} \int_{\O} \int_{\O} \r_{\d} (\vec x' {-} \vec x) \left[ \left( \mc{E} (\vec u, \vec P) (\vec x, \vec x') - \frac{1}{n} \div \vec u (\vec x) \right)^2 \right] \dd \vec x' \, \dd \vec x \\
 = & n \int_{\O} \dashint_{\Sn} \left( (\nabla \vec u (\vec x) - \vec P (\vec x)) \vec z \cdot \vec z - \frac{1}{n} \div \vec u (\vec x) \right)^2 \, \dd \Hn (\vec z) \, \dd \vec x .
\end{split}
\end{equation}
We express
\begin{equation}\label{eq:limFE2}
\begin{split}
 &\int_{\O} \int_{\O} \r_{\d} (\vec x' {-} \vec x) \mc{E} (\vec u, \vec P) (\vec x, \vec x')^2 \, \dd \vec x' \, \dd \vec x \\
 = & \int_{\O} \int_{\O} \r_{\d} (\vec x' {-} \vec x) \left( \frac{(\nabla \vec u (\vec x) - \vec P (\vec x)) (\vec x' {-} \vec x) \cdot (\vec x' {-} \vec x)}{|\vec x' {-} \vec x|^2} \right)^2 \, \dd \vec x' \, \dd \vec x + \int_{\O} \int_{\O} \r_{\d} (\vec x' {-} \vec x) C (\vec x, \vec x') \, \dd \vec x' \, \dd \vec x
\end{split}
\end{equation}
with
\begin{align*}
 C (\vec x, \vec x') = & \frac{\left( \vec u (\vec x') - \vec u (\vec x) - \nabla \vec u (\vec x) (\vec x' {-} \vec x) \right) \cdot (\vec x' {-} \vec x)}{|\vec x' {-} \vec x|^2} \\
 &  \times \left( \frac{\left( \vec u (\vec x') - \vec u (\vec x) - \nabla \vec u (\vec x) (\vec x' {-} \vec x) \right) \cdot (\vec x' {-} \vec x)}{|\vec x' {-} \vec x|^2} + 2 \frac{(\nabla \vec u (\vec x) - \vec P (\vec x) )(\vec x' {-} \vec x) \cdot (\vec x' {-} \vec x)}{|\vec x' {-} \vec x|^2}\right) .
\end{align*}
We have, thanks to \eqref{eq:diffgrad},
\[
 |C (\vec x, \vec x')| \leq \s (|\vec x' {-} \vec x|) \left( \| \s \|_{\infty} + 2 \| \nabla \vec u \|_{\infty} + 2 \| \vec P \|_{\infty} \right) ,
\]
so for a.e.\ $\vec x \in \O$,
\begin{equation}\label{eq:C1}
 \left| \int_{\O} \r_{\d} (\vec x' {-} \vec x) C (\vec x, \vec x') \, \dd \vec x' \right| \leq  \left( \| \s \|_{\infty} + 2 \| \nabla \vec u \|_{\infty} + 2 \| \vec P \|_{\infty} \right) \int_{\O - \vec x} \r_{\d} (\tilde{\vec x}) \, \s (|\tilde{\vec x}|) \, \dd \tilde{\vec x} 
\end{equation}
and, for any $r > 0$,
\begin{equation}\label{eq:C2}
 \int_{\O - \vec x} \r_{\d} (\tilde{\vec x}) \, \s (|\tilde{\vec x}|) \, \dd \tilde{\vec x} \, \dd \vec x \leq  \int_{\Rn} \r_{\d} (\tilde{\vec x}) \, \s (|\tilde{\vec x}|) \, \dd \tilde{\vec x} \leq n \s (r) + \| \s \|_{\infty} \int_{\Rn \setminus B (\vec 0, r)} \r_{\d} (\tilde{\vec x}) \, \dd \tilde{\vec x} .
\end{equation}
Bounds \eqref{eq:C1} and \eqref{eq:C2}, as well as properties \eqref{newR2}
 and \eqref{eq:lims0}, imply that
\begin{equation}\label{eq:limF0}
 \lim_{\d \to 0} \int_{\O} \int_{\O} \r_{\d} (\vec x' {-} \vec x) C (\vec x, \vec x') \, \dd \vec x' \, \dd \vec x = 0 .
\end{equation}
Now let $A \ssubset \O$ be measurable and $0 < r < \dist (A, \p \O)$.
Then, for any $\vec x \in A$,
\begin{equation}\label{eq:limF1}
\begin{split}
 & \int_{\O} \r_{\d} (\vec x' {-} \vec x) \left( \frac{(\nabla \vec u (\vec x) - \vec P (\vec x)) (\vec x' {-} \vec x) \cdot (\vec x' {-} \vec x)}{|\vec x' {-} \vec x|^2} \right)^2 \, \dd \vec x' \\
 & = \left[ \int_{B (\vec 0, r)} + \int_{(\O - \vec x) \setminus B (\vec 0, r)} \right] \r_{\d} (\tilde{\vec x}) \left( \frac{(\nabla \vec u (\vec x) - \vec P (\vec x)) \tilde{\vec x} \cdot \tilde{\vec x}}{|\tilde{\vec x}|^2} \right)^2 \, \dd \tilde{\vec x} ,
\end{split}
\end{equation}
with, thanks to Lemma \ref{le:formulasB},
\begin{equation}\label{eq:limF2}
 \int_{B (\vec 0, r)} \r_{\d} (\tilde{\vec x}) \left( \frac{(\nabla \vec u (\vec x) - \vec P (\vec x)) \tilde{\vec x} \cdot \tilde{\vec x}}{|\tilde{\vec x}|^2} \right)^2 \dd \tilde{\vec x} = \int_{B (\vec 0, r)} \r_{\d} (\tilde{\vec x}) \, \dd \tilde{\vec x} \ \dashint_{\Sn} \left( (\nabla \vec u (\vec x) - \vec P (\vec x)) \vec z \cdot \vec z \right)^2  \dd \Hn (\vec z)
\end{equation}
and
\begin{equation}\label{eq:limF3}
 \left| \int_{(\O - \vec x) \setminus B (\vec 0, r)} \r_{\d} (\tilde{\vec x}) \left( \frac{(\nabla \vec u (\vec x) - \vec P (\vec x)) \tilde{\vec x} \cdot \tilde{\vec x}}{|\tilde{\vec x}|^2} \right)^2 \, \dd \tilde{\vec x} \right| \leq 2 \left( \| \nabla \vec u \|_{\infty}^2 + \| \vec P \|_{\infty}^2 \right) \int_{\Rn \setminus B (\vec 0, r)} \r_{\d} (\tilde{\vec x} )\, \dd \tilde{\vec x} .
\end{equation}
Note that the bound \eqref{eq:boundEgrad} implies that the family of functions
\[
 \vec x \mapsto \int_{\O} \r_{\d} (\vec x' {-} \vec x) \mc{E} (\vec u, \vec P) (\vec x, \vec x')^2 \, \dd \vec x'
\]
is equiintegrable in $\O$ for $\d>0$.
Hence, property \eqref{newR2}, together with bound \eqref{eq:limF3} and equalities \eqref{eq:limF1}--\eqref{eq:limF2} show that
\begin{align*}
 & \lim_{\d \to 0} \int_{\O} \int_{\O} \r_{\d} (\vec x' {-} \vec x) \left( \frac{(\nabla \vec u (\vec x) - \vec P (\vec x)) (\vec x' {-} \vec x) \cdot (\vec x' {-} \vec x)}{|\vec x' {-} \vec x|^2} \right)^2 \dd \vec x' \, \dd \vec x \\
 = &  n \int_{\O} \dashint_{\Sn} \left( (\nabla \vec u (\vec x) - \vec P (\vec x)) \vec z \cdot \vec z \right)^2 \dd \Hn (\vec z) \, \dd \vec x ,
\end{align*}
which, together with \eqref{eq:limFE2} and \eqref{eq:limF0}, implies
\begin{equation}\label{eq:Edivfirst}
 \lim_{\d \to 0} \int_{\O} \int_{\O} \r_{\d} (\vec x' {-} \vec x) \mc{E} (\vec u, \vec P) (\vec x, \vec x')^2 \, \dd \vec x' \, \dd \vec x = n \int_{\O} \dashint_{\Sn} \left( (\nabla \vec u (\vec x) - \vec P (\vec x)) \vec z \cdot \vec z \right)^2 \dd \Hn (\vec z) \, \dd \vec x.
\end{equation}

Now we express
\begin{equation}\label{eq:limFE2second}
\begin{split}
 & \int_{\O} \int_{\O} \r_{\d} (\vec x' {-} \vec x) \mc{E} (\vec u, \vec P) (\vec x, \vec x') \div \vec u (\vec x) \dd \vec x' \, \dd \vec x \\
 = & \int_{\O} \int_{\O} \r_{\d} (\vec x' {-} \vec x) \frac{(\nabla \vec u (\vec x) - \vec P (\vec x)) (\vec x' {-} \vec x) \cdot (\vec x' {-} \vec x)}{|\vec x' {-} \vec x|^2} \div \vec u (\vec x) \, \dd \vec x' \, \dd \vec x + \int_{\O} \int_{\O} \r_{\d} (\vec x' {-} \vec x) B (\vec x, \vec x') \, \dd \vec x' \, \dd \vec x
 \end{split}
\end{equation}
with
\begin{equation*}
 B (\vec x, \vec x') = \frac{\left( \vec u (\vec x') - \vec u (\vec x) - \nabla \vec u (\vec x) (\vec x' {-} \vec x) \right) \cdot (\vec x' {-} \vec x)}{|\vec x' {-} \vec x|^2} \div \vec u (\vec x) .
\end{equation*}
We have, thanks to \eqref{eq:diffgrad},
\[
 |B (\vec x, \vec x')| \leq \s (|\vec x' {-} \vec x|) \| \div \vec u \|_{\infty} .
\]
An analogous reasoning to that of \eqref{eq:C1}, \eqref{eq:C2} and \eqref {eq:limF0} leads to
\begin{equation}\label{eq:limF0second}
 \lim_{\d \to 0} \int_{\O} \int_{\O} \r_{\d} (\vec x' {-} \vec x) B (\vec x, \vec x') \, \dd \vec x' \, \dd \vec x = 0 .
\end{equation}
Now let $A \ssubset \O$ be measurable and $0 < r < \dist (A, \p \O)$.
Then, for any $\vec x \in A$,
\begin{equation}\label{eq:limF1second}
\begin{split}
 & \int_{\O} \r_{\d} (\vec x' {-} \vec x) \frac{(\nabla \vec u (\vec x) - \vec P (\vec x)) (\vec x' {-} \vec x) \cdot (\vec x' {-} \vec x)}{|\vec x' {-} \vec x|^2} \div \vec u (\vec x) \, \dd \vec x' \\
 & = \left[ \int_{B (\vec 0, r)} + \int_{(\O - \vec x) \setminus B (\vec 0, r)} \right] \r_{\d} (\tilde{\vec x}) \frac{(\nabla \vec u (\vec x) - \vec P (\vec x)) \tilde{\vec x} \cdot \tilde{\vec x}}{|\tilde{\vec x}|^2} \div \vec u (\vec x) \, \dd \tilde{\vec x} ,
\end{split}
\end{equation}
with, thanks to Lemma \ref{le:formulasB},
\begin{equation}\label{eq:limF2second}
 \int_{B (\vec 0, r)} \r_{\d} (\tilde{\vec x}) \frac{(\nabla \vec u (\vec x) - \vec P (\vec x)) \tilde{\vec x} \cdot \tilde{\vec x}}{|\tilde{\vec x}|^2} \div \vec u (\vec x) \, \dd \tilde{\vec x} = \int_{B (\vec 0, r)} \r_{\d} (\tilde{\vec x}) \, \dd \tilde{\vec x} \ \dashint_{\Sn} (\nabla \vec u (\vec x) - \vec P (\vec x)) \vec z \cdot \vec z \div \vec u (\vec x) \, \dd \Hn (\vec z)
\end{equation}
and
\begin{equation}\label{eq:limF3second}
 \left| \int_{(\O - \vec x) \setminus B (\vec 0, r)} \r_{\d} (\tilde{\vec x}) \frac{(\nabla \vec u (\vec x) - \vec P (\vec x)) \tilde{\vec x} \cdot \tilde{\vec x}}{|\tilde{\vec x}|^2} \div \vec u (\vec x) \, \dd \tilde{\vec x} \right| \leq \| \div \vec u \|_{\infty} \left( \| \nabla \vec u \|_{\infty} + \| \vec P \|_{\infty} \right) \int_{\Rn \setminus B (\vec 0, r)} \r_{\d} (\tilde{\vec x} )\, \dd \tilde{\vec x} .
\end{equation}
Note that the bound \eqref{eq:boundEgrad} implies that the family of functions
\[
 \vec x \mapsto \int_{\O} \r_{\d} (\vec x' {-} \vec x) \mc{E} (\vec u, \vec P) (\vec x, \vec x') \div \vec u (\vec x) \, \dd \vec x'
\]
is equiintegrable in $\O$ for $\d>0$.
Hence, property \eqref{newR2}, together with bound \eqref{eq:limF3second} and equalities \eqref{eq:limF1second}--\eqref{eq:limF2second} show that
\begin{align*}
 & \lim_{\d \to 0} \int_{\O} \int_{\O} \r_{\d} (\vec x' {-} \vec x) \frac{(\nabla \vec u (\vec x) - \vec P (\vec x)) (\vec x' {-} \vec x) \cdot (\vec x' {-} \vec x)}{|\vec x' {-} \vec x|^2} \div \vec u (\vec x) \, \dd \vec x' \, \dd \vec x \\
 = &  n \int_{\O} \dashint_{\Sn} (\nabla \vec u (\vec x) - \vec P (\vec x)) \vec z \cdot \vec z \, \dd \Hn (\vec z) \div \vec u (\vec x) \, \dd \vec x ,
\end{align*}
which, together with \eqref{eq:limFE2second} and \eqref{eq:limF0second}, implies
\begin{equation}\label{eq:Edivsecond}
 \lim_{\d \to 0} \int_{\O} \int_{\O} \r_{\d} (\vec x' {-} \vec x) \mc{E} (\vec u, \vec P) (\vec x, \vec x') \div \vec u (\vec x) \, \dd \vec x' \, \dd \vec x = n \int_{\O} \dashint_{\Sn} (\nabla \vec u (\vec x) - \vec P (\vec x)) \vec z \cdot \vec z \, \dd \Hn (\vec z) \div \vec u (\vec x) \, \dd \vec x.
\end{equation}

Now let $A \ssubset \O$ be measurable and $0 < r < \dist (A, \p \O)$.
Then, for any $\vec x \in A$,
\begin{equation}\label{eq:limF1third}
 \int_{\O} \r_{\d} (\vec x' {-} \vec x) \div \vec u (\vec x)^2 \, \dd \vec x' = \left[ \int_{B (\vec 0, r)} + \int_{(\O - \vec x) \setminus B (\vec 0, r)} \right] \r_{\d} (\tilde{\vec x}) \div \vec u (\vec x)^2 \, \dd \tilde{\vec x} ,
\end{equation}
with
\begin{equation}\label{eq:limF3secondagain}
 \left| \int_{(\O - \vec x) \setminus B (\vec 0, r)} \r_{\d} (\tilde{\vec x}) \div \vec u (\vec x)^2 \, \dd \tilde{\vec x} \right| \leq \| \div \vec u \|_{\infty}^2 \int_{\Rn \setminus B (\vec 0, r)} \r_{\d} (\tilde{\vec x} )\, \dd \tilde{\vec x} .
\end{equation}
Note that the bound
\[
 \int_{\O} \r_{\d} (\vec x' {-} \vec x) \div \vec u (\vec x)^2 \, \dd \vec x' \leq n \| \div \vec u \|_{\infty}^2
\]
implies that the family of functions
\[
 \vec x \mapsto \int_{\O} \r_{\d} (\vec x' {-} \vec x) \div \vec u (\vec x)^2 \, \dd \vec x'
\]
is equiintegrable in $\O$ for $\d>0$.
Hence, property \eqref{newR2}, together with bound \eqref{eq:limF3secondagain} and equality \eqref{eq:limF1third} show that
\begin{equation}\label{eq:Edivthird}
 \lim_{\d \to 0} \int_{\O} \int_{\O} \r_{\d} (\vec x' {-} \vec x) \div \vec u (\vec x)^2 \, \dd \vec x' \, \dd \vec x = n \int_{\O} \div \vec u (\vec x)^2 \, \dd \vec x ,
\end{equation}

Equalities \eqref{eq:Edivfirst}, \eqref{eq:Edivsecond} and \eqref{eq:Edivthird} show \eqref{eq:divudivu}, while \eqref{eq:divudivu} and \eqref{eq:Ed-div} yield \eqref{eq:limitE-Ed} and complete the proof of this step.

\smallskip

\emph{Step 2.} Now we just assume $\vec u \in H^1 (\O, \Rn)$ and $\vec P \in L^2 (\O, \Rnnsd)$, as in the statement.
Let $\e>0$ and let $\bar{\vec u} \in C^1 (\bar{\O}, \Rn)$ and $\bar{\vec P} \in C (\bar{\O}, \Rnnsd)$ be such that
\[
 \left\| \bar{\vec u} - \vec u \right\|_{H^1} \leq \e \quad \text{and} \quad \left\| \bar{\vec P} - \vec P \right\|_2 \leq \e .
\]
This is possible since $\Rnnsd$ is a subspace of $\Rnn$.

Now, consider Lemma \ref{le:linearbounded} and the operator defined therein, which we call $T_{\d}$ in order to underline the dependence on $\d$.
By Lemmas \ref{le:Tr}, \ref{le:linearbounded} and \ref{le:SrW12} there exists $C>0$ independent of $\d$ such that $\| T_{\d} (\vec v, \vec Q) \|_{L^2 (\O \times \O)} \leq C \left( \| \vec v \|_{H^1} + \| \vec Q \|_2 \right)$ for all $\vec v \in H^1 (\O, \Rn)$ and $\vec Q \in L^2 (\O, \Rnnsd)$.
Then,
\begin{align*}
 & \left| \int_{\O} \int_{\O} \r_{\d} (\vec x' {-} \vec x) \left[ \left( \mc{E} (\bar{\vec u}, \bar{\vec P}) (\vec x, \vec x') - \frac{1}{n} \mf{E}_{\d} (\bar{\vec u}, \bar{\vec P}) (\vec x) \right)^2 - \left( \mc{E} (\vec u, \vec P) (\vec x, \vec x') - \frac{1}{n} \mf{E}_{\d} (\vec u, \vec P) (\vec x) \right)^2 \right] \dd \vec x' \, \dd \vec x \right| \\
 = & \left| \left\| T_{\d} (\bar{\vec u}, \bar{\vec P}) \right\|_{L^2 (\O \times \O)}^2 - \left\| T_{\d} (\vec u, \vec P) \right\|_{L^2 (\O \times \O)}^2 \right| \leq \left\| T_{\d} (\bar{\vec u} - \vec u, \bar{\vec P} - \vec P)  \right\|_{L^2 (\O \times \O)} \left\| T_{\d} (\bar{\vec u} + \vec u, \bar{\vec P} + \vec P)  \right\|_{L^2 (\O \times \O)} \\
 \leq & 4 \, C^2 \, \e \left( \e + \left\| \vec u \right\|_{H^1} + \left\| \vec P \right\|_2 \right).
\end{align*}
This concludes the proof.
\end{proof}

\begin{lemma}[Convergence of $F_{\d}$ along smooth sequences]\label{le:limitFdupNEW}
Let $A \subset \O$ be a Lipschitz domain.
For each $\d>0$ let $\vec u_{\d}, \vec u \in C^1 (\bar{A}, \Rn)$, $\vec P_{\d}, \vec P \in C (\bar{A}, \Rnnsd)$ and $d_{\d} \in C (\bar{A})$ satisfy
\[
 \vec u_{\d} \to \vec u \text{ in } C^1 (\bar{A}, \Rn), \qquad \vec P_{\d} \to \vec P \text{ in } C (\bar{A}, \Rnnsd) \quad \text{and} \quad d_{\d} \to \div \vec u \text{ in } C (\bar{A}) \quad \text{as } \d \to 0.
\]
Then 
\begin{align*}
 & \lim_{\d \to 0} \int_A \int_A \r_{\d} (\vec x' {-} \vec x) \left( \mc{E} (\vec u_{\d}, \vec P_{\d}) (\vec x, \vec x') - d_{\d}(\vec x) \right)^2 \dd \vec x' \, \dd \vec x \\
 & = n \int_A \dashint_{\Sn} \left( \left( \nabla \vec u (\vec x) - \vec P (\vec x) \right) \vec z \cdot \vec z - \div \vec u (\vec x) \right)^2 \dd \Hn (\vec z) \, \dd \vec x .
\end{align*}
\end{lemma}
\begin{proof}
We have
\[
 \left( \mc{E} (\vec u_{\d}, \vec P_{\d}) - d_{\d}) \right)^2 - \left( \mc{E} (\vec u, \vec P) - \div \vec u \right)^2 = \left[  \mc{E} (\vec u_{\d} - \vec u, \vec P_{\d} - \vec P) + \div \vec u - d_{\d} \right] \left[ \mc{E} (\vec u_{\d} + \vec u, \vec P_{\d} + \vec P) - d_{\d} - \div \vec u \right] .
\]
We now use estimates \eqref{eq:boundEgrad} to infer that
\[
 \lim_{\d \to 0} \left\| \left( \mc{E} (\vec u_{\d}, \vec P_{\d}) - d_{\d} \right)^2 - \left( \mc{E} (\vec u, \vec P) - \div \vec u \right)^2 \right\|_{\infty} = 0 .
\]
Then, by uniform convergence and equality \eqref{eq:limitE-Ed} we conclude
\begin{align*}
 & \lim_{\d \to 0} \int_A \int_A \r_{\d} (\vec x' {-} \vec x) \left( \mc{E} (\vec u_{\d}, \vec P_{\d}) (\vec x, \vec x') - d_{\d}(\vec x) \right)^2 \dd \vec x' \, \dd \vec x \\
 = & \lim_{\d \to 0} \int_A \int_A \r_{\d} (\vec x' {-} \vec x) \left( \mc{E} (\vec u, \vec P) (\vec x, \vec x') - \div \vec u (\vec x) \right)^2 \dd \vec x' \, \dd \vec x \\
 = & n \int_A \dashint_{\Sn} \left( \left( \nabla \vec u (\vec x) - \vec P (\vec x) \right) \vec z \cdot \vec z - \div \vec u (\vec x) \right)^2 \dd \Hn (\vec z) \, \dd \vec x ,
\end{align*}
as desired.
\end{proof}

The following nonlocal Korn inequality of \cite[Lemma 4.4]{MeDu15}, with a constant independent of $\d$, is essential in the proof of the $\G$-convergence.
\begin{proposition}[Uniform nonlocal Korn inequality]\label{prop:Kornd}
Let $\{ \r_{\d} \}_{\d>0}$ be a family of kernels satisfying \eqref{newR}--\eqref{newR2}.
Then there exist $C>0$ and $\d_0>0$ such that for all $0 < \d < \d_0$ and $\vec u \in V \cap \mc{S}_{\d} (\O)$,
\[
 \left\| \vec u \right\|_2^2 \leq C \left| \vec u \right|_{\mc{S}_{\d}}^2 .
\]
\end{proposition}

With Proposition \ref{prop:Kornd} at hand, we can show the following coercivity bound for $F_{\d}$.

\begin{lemma}[Uniform coercivity of the energy]\label{le:Fdcoercivity}
Let $\{ \r_{\d} \}_{\d>0}$ be a family of kernels satisfying \eqref{newR}--\eqref{newR2}.
Then there exist $c>0$ and $\d_0>0$ such that for all $0 < \d < \d_0$ and $(\vec u, \vec P) \in (V \cap \mc{S}_{\d} (\O)) \times L^2 (\O, \Rnnsd)$,
\[
 F_{\d} (\vec u, \vec P) \geq c \|(\vec u, \vec P)\|_{\mc{T}_{\d}}^2 -\frac{1}{c} .
\]
\end{lemma}
\begin{proof}
We repeat the proof of Lemma \ref{le:Fcoerc} until \eqref{eq:Fgeq}: we then find that there exists $c_1>0$ such that for all $\d>0$, all $(\vec u, \vec P) \in \mc{T}_{\d}$ and all $\eta > 0$,
\[
 F_{\d} (\vec u, \vec P) \geq c_1 \left( |(\vec u, \vec P)|_{\mc{T}_{\d}}^2 + \|\vec P\|_2^2 \right) + c_1 \left( |(\vec u, \vec P)|_{\mc{T}_{\d}}^2 + \|\vec P\|_2^2 \right) - \frac{\eta}{2} \|\vec u\|_2^2 - \frac{1}{2\eta} \| \vec b \|_2^2 .
\]
By Proposition \ref{prop:Kornd} and estimate \eqref{eq:uSupT}, there exist $C>0$ and $\d_0>0$ such that for all $0 < \d < \d_0$,
\begin{equation*}
  \| \vec u \|_2^2 \leq C |\vec u|_{\mc{S}_{\d}}^2 \leq 2 n C \left( |(\vec u, \vec P)|_{\mc{T}_{\d}}^2 + \| \vec P \|_2^2 \right).
\end{equation*}
Putting together both inequalities, we find that
\[
 F_{\d} (\vec u, \vec P) \geq c_1 \left( |(\vec u, \vec P)|_{\mc{T}_{\d}}^2 + \|\vec P\|_2^2 \right) + \left( \frac{c_1}{2 n C} - \frac{\eta}{2} \right) \|\vec u\|_2^2 - \frac{1}{2\eta} \| \vec b \|_2^2 .
\]
Choosing $\eta>0$ such that
\[
 \frac{c_1}{2 n C} - \frac{\eta}{2} > 0
\]
concludes the proof.
\end{proof}

We present the fundamental compactness result of  \cite[Prop.\ 4.2]{MeDu15}.

\begin{proposition}[Compactness]\label{prop:compactness}
Let $\{ \r_{\d} \}_{\d>0}$ be a  sequence  of kernels satisfying \eqref{newR}--\eqref{newR2}.
Let $\{ \vec u_{\d} \}_{\d>0}$ be a  sequence  in $L^2 (\O, \Rn)$ satisfying
\[
 \sup_{\d>0} \left\| \vec u_{\d} \right\|_{\mc{S}_{\d}} < \infty .
\]
Then there exists a decreasing sequence $\d_j \to 0$ and a $\vec u \in L^2 (\O, \Rn)$ such that $\vec u_{\d_j} \to \vec u$ in $L^2 (\O, \Rn)$.
Moreover, for any such sequence and any such $\vec u$ we have that $\vec u \in H^1 (\O, \Rn)$.
\end{proposition}

We now have all ingredients to prove the $\G$-limit result.
As usual, we divide it into three parts: compactness, lower bound and upper bound.
We label the sequences with $\d$, the same parameter of $F_{\d}$, and, of course, it is implicit that $\d \to 0$.

\begin{theorem}[$\Gamma$-convergence of the energy]\label{thm:gamma}
Let $\mc{V}_{\d} = (V\cap \mc{S}_{\d}) \times L^2 (\O, \Rnnsd)$.
\begin{enumerate}[a)]
\item\label{item:Gammaa} Let $(\vec u_{\d}, \vec P_{\d}) \in \mc{V}_{\d}$ satisfy $\sup_{\d} F_{\d} (\vec u_{\d}, \vec P_{\d}) < \infty$.
Then there exists  $(\vec u, \vec P) \in H^1 (\O, \Rn) \times L^2 (\O, \Rnnsd)$ such that, for a subsequence, $\vec u_{\d} \to \vec u$ in $L^2 (\O, \Rn)$ and $\vec P_{\d} \weakc \vec P$ in $L^2 (\O, \Rnnsd)$.

\item\label{item:Gammab} Let $(\vec u_{\d}, \vec P_{\d}) \in \mc{V}_{\d}$ and $(\vec u, \vec P) \in H^1 (\O, \Rn) \times L^2 (\O, \Rnnsd)$ satisfy $\vec u_{\d} \to \vec u$ in $L^2 (\O, \Rn)$ and $\vec P_{\d} \weakc \vec P$ in $L^2 (\O, \Rnnsd)$.
Then
\[
 F_0 (\vec u, \vec P) \leq \liminf_{\d \to 0} F_{\d} (\vec u_{\d}, \vec P_{\d}) .
\]

\item\label{item:Gammac} Let $(\vec u, \vec P) \in (V \cap H^1 (\O, \Rn)) \times L^2 (\O, \Rnnsd)$.
Then for each $\d$ there exists $(\vec u_{\d}, \vec P_{\d}) \in \mc{V}_{\d}$ such that  
\[
 F_0 (\vec u, \vec P) = \lim_{\d \to 0} F_{\d} (\vec u_{\d}, \vec P_{\d}) .
\]
\end{enumerate}
\end{theorem}
\begin{proof}
Part \emph{\ref{item:Gammaa})}.
By Lemma \ref{le:Fdcoercivity}, the set $\{ \left\| (\vec u, \vec P) \right\|_{\mc{T}_{\d}}^2 \}_{\d>0}$ is bounded.
We then apply Proposition \ref{prop:compactness} to find the existence of $\vec u$, and the boundedness of $\{ \vec P_{\d} \}_{\d>0}$ in $L^2 (\O, \Rnnsd)$ for the existence of $\vec P$.

Part \emph{\ref{item:Gammab})}.
Clearly,
\[
 \| \vec P \|_2^2 \leq \liminf_{\d \to 0} \| \vec P_{\d} \|_2^2 \quad \text{and} \quad \lim_{\d \to 0} \int_{\O} \vec b (\vec x) \cdot \vec u_{\d} (\vec x) \, \dd \vec x = \int_{\O} \vec b (\vec x) \cdot \vec u (\vec x) \, \dd \vec x .
\]
Moreover, as mentioned in Lemma \ref{le:Eddiv}, it was proved in \cite[Lemma 3.6]{MeDu15} that $\mf{D}_{\d} (\vec u_{\d}) \weakc \div \vec u$ in $L^2 (\O)$ as $\d \to 0$, so  $\| \div \vec u \|_2^2 \leq \liminf_{\d} \| \mf{D}_{\d} (\vec u_{\d}) \|_2^2$.
Hence, we are left to the analysis of the remaining term.

Let $\{ \f_r \}_{r>0}$ be the family of mollifiers defined in Appendix \ref{se:auxiliary}.
Let $A \ssubset \O$ be a Lipschitz domain and let $0 < r < \dist (A, \p \O)$.
By Lemma \ref{le:Jensen},
\begin{equation}\label{eq:lb1}
\begin{split}
 & \int_A \int_A \r_{\d} (\vec x {-} \vec x') \left( \mc{E} (\f_r \star \vec u_{\d}, \f_r \star \vec P_{\d}) (\vec x, \vec x') - \frac{1}{n} \f_r \star \mf{E}_{\d} (\vec u_{\d}, \vec P_{\d}) (\vec x) \right)^2 \dd \vec x' \, \dd \vec x \\
 \leq & \int_{\O} \int_{\O} \r_{\d} (\vec x {-} \vec x') \left( \mc{E} (\vec u_{\d}, \vec P_{\d}) (\vec x, \vec x') - \frac{1}{n} \mf{E}_{\d} (\vec u_{\d}, \vec P_{\d}) (\vec x) \right)^2 \dd \vec x' \, \dd \vec x .
\end{split}
\end{equation}
Call $\vec u_r = \f_r \star \vec u$ and $\vec P_r = \f_r \star \vec P$.
Standard properties of mollifiers show that $\f_r \star \vec u_{\d} \to \vec u_r$ in $C^1 (\bar{A}, \Rn)$ and $\f_r \star \vec P_{\d, r} \to \vec P_r$ in $C (\bar{A}, \Rnnsd)$ as $\d \to 0$.
Using also Lemma \ref{le:Eddiv}, we find that $\f_r \star \mf{E}_{\d} (\vec u_{\d}, \vec P_{\d}) \to \div \vec u_r$ in $C (\bar{A})$ as $\d \to 0$.
Thus, letting $\d \to 0$ in \eqref{eq:lb1} and using Lemma \ref{le:limitFdupNEW}, we obtain
\begin{equation}\label{eq:lb2}
\begin{split}
 & n \int_A \dashint_{\Sn} \left( \left( \nabla \vec u_r (\vec x) - \vec P_r (\vec x) \right) \vec z \cdot \vec z - \frac{1}{n} \div \vec u_r (\vec x) \right)^2 \dd \Hn (\vec z) \, \dd \vec x \\
 \leq & \liminf_{\d \to 0} \int_{\O} \int_{\O} \r_{\d} (\vec x' {-} \vec x) \left( \mc{E} (\vec u_{\d}, \vec P_{\d}) (\vec x, \vec x') - \frac{1}{n} \mf{E}_{\d} (\vec u_{\d}, \vec P_{\d}) (\vec x) \right)^2 \dd \vec x' \, \dd \vec x .
\end{split}
\end{equation}
Again, standard properties of mollifiers show that $\nabla \vec u_r \to \nabla \vec u$ in $L^2 (A, \Rnn)$ and a.e., and $\vec P_r \to \vec P$ in $L^2 (A, \Rnnsd)$ and a.e., as $r \to 0$.
We then let $r\to 0$ and apply dominated convergence in \eqref{eq:lb2} to get 
\begin{equation}\label{eq:lb3}
\begin{split}
 & n \int_A \dashint_{\Sn} \left( (\nabla \vec u (\vec x) - \vec P (\vec x)) \vec z \cdot \vec z - \frac{1}{n} \div \vec u (\vec x) \right)^2 \, \dd \Hn (\vec z) \, \dd \vec x \\
 \leq & \liminf_{\d \to 0} \int_{\O} \int_{\O} \r_{\d} (\vec x' {-} \vec x) \left( \mc{E} (\vec u_{\d}, \vec P_{\d}) (\vec x, \vec x') - \frac{1}{n} \mf{E}_{\d} (\vec u_{\d}, \vec P_{\d}) (\vec x) \right)^2 \dd \vec x' \, \dd \vec x .
\end{split}
\end{equation}
Finally, we send $A \nearrow \O$ and use monotone convergence in
\eqref{eq:lb3} to obtain
\begin{align*}
 & n \int_{\O} \dashint_{\Sn} \left( (\nabla \vec u (\vec x) - \vec P (\vec x)) \vec z \cdot \vec z - \frac{1}{n} \div \vec u (\vec x) \right)^2 \, \dd \Hn (\vec z) \, \dd \vec x \\
 \leq & \liminf_{\d \to 0} \int_{\O} \int_{\O} \r_{\d} (\vec x' {-} \vec x) \left( \mc{E} (\vec u_{\d}, \vec P_{\d}) (\vec x, \vec x') - \frac{1}{n} \mf{E}_{\d} (\vec u_{\d}, \vec P_{\d}) (\vec x) \right)^2 \dd \vec x' \, \dd \vec x .
\end{align*}

Part \emph{\ref{item:Gammac})}.
This follows from Proposition \ref{prop:limitFdup} by taking $(\vec u_{\d}, \vec P_{\d}) = (\vec u, \vec P)$. 
\end{proof}

 We are now ready to present the small-horizon convergence result
for the incremental problem.

\begin{corollary}[Convergence to the local incremental problem]
  Let 
$\vec P_{\rm old} \in L^2 (\O,
  \Rnnsd)$ be given and  $(\vec u_{\d}, \vec P_{\d})$ be the solution of the
  nonlocal incremental problem \eqref{eq:incremental}. Then $(\vec u_{\d},
  \vec P_{\d}) \to (\vec u_0,\vec P_0)$ with the respect to the strong
  $\times$ weak topology in $Q$,
  where $ (\vec u_0,\vec P_0) \in (V \cap H^1(\Omega,\Rn)) \times  L^2(\O, \Rnnsd)$ is the solution of the local
  incremental problem \eqref{eq:incremental_local}.
\end{corollary}
\begin{proof}
  For each $\d>0$ we have
  \[
 F_{\d} (\vec u_{\d}, \vec P_{\d}) + H(\vec P_{\d} - \vec P_{\rm old}) \leq F_{\d} (\vec u_0, \vec P_0) + H(\vec P_0 - \vec P_{\rm old}) ,
\]
so by Proposition \ref{prop:limitFdup}, $\sup_{\d>0} F_{\d} (\vec u_{\d}, \vec P_{\d}) < \infty$.
By Theorem \ref{thm:gamma}, the sequence $(\vec u_{\d}, \vec P_{\d})$ is precompact in the strong $\times$ weak topology in $Q$.
Thus, one is left to prove the
  $\Gamma$-convergence of $F_{\d} + H(\cdot {-}\vec P_{\rm old})$ as
  $\d\to 0$. The $\Gamma$-$\liminf$ follows from the $\Gamma$-convergence of
  $F_{\d}$ in Theorem \ref{thm:gamma} as $H$
  is independent of $\d$ and lower semicontinuous. The existence of a
  recovery sequence follows by pointwise convergence: see Proposition \ref{prop:limitFdup}.
\end{proof}

\section{ Quasistatic evolution}\label{se:quasistatic}

Assume now that the body force $\vec b $ depends on time, namely let  $\vec
b \in W^{1,1}(0,T;L^2(\O;\Rn))$. Correspondingly, without introducing
new notation, we indicate the time-dependent (complementary) energy of
the medium via
$F_\r: Q\times [0,T] \to \R\cup \{\infty\}$
given by
\begin{align*}
F_{\r} (\vec u, \vec P,t) = & \b \int_{\O} \mf{D}_{\r} (\vec u) (\vec x)^2 \, \dd \vec x + \a \int_{\O} \int_{\O} \r (\vec x' {-} \vec x) \left( \mc{E} (\vec u, \vec P) (\vec x, \vec x') - \frac{1}{n} \mf{E}_{\r} (\vec u, \vec P) (\vec x) \right)^2 \dd \vec x' \, \dd \vec x \\
 & +
 \g \int_{\O} |\vec P (\vec x)|^2 \, \dd \vec x - \int_{\O} \vec b (\vec x,t) \cdot \vec u (\vec x) \, \dd \vec x .
\end{align*}
Note that boundary conditions could be taken to be time dependent as
well by letting $\vec u - \vec u_{\rm Dir}(t) \in V$ where $ \vec
u_{\rm Dir}(t)$ is given. This would originate an additional time-dependent
linear term in the energy. We, however, stick to the time-independent condition $\vec u \in V$, for the sake of simplicity.

The quasistatic
elastoplastic evolution of the medium \eqref{var100}--\eqref{var200}
can be then specified as 
\begin{align}
  \partial_{\vec u} F_\r (\vec u(t), \vec P(t),t) &= \vec 0 \ \  
  \ \text{in} \ \ \mc{S}_\r^*, \label{var1}\\
\partial_{\dot{\vec P}} H(\dot{\vec P}(t)) + \partial_{\vec P} F_\r
  (\vec u(t), \vec P(t),t) &\ni \vec 0  \ \  
  \ \text{in} \ \ L^2(\Omega;\Rnnsd).  \label{var2}
\end{align}
We have denoted by $\mc{S}_\r^*$ the dual of $\mc{S}_\r$.
In particular, relation \eqref{var2} is a pointwise-in-time
inclusion in $L^2(\O,\Rnnsd)$.

System \eqref{var1}--\eqref{var2} can be made more
explicit by introducing the bilinear form $B_\r$
associated to the quadratic part of $F_\r$, namely,
\begin{align*}
  & B_\r((\vec u, \vec P), (\vec v, \vec Q))= 
\beta \int_\O \mf{D}_\d  (\vec u)(\vec x) \, \mf{D}_\d (\vec v)(\vec x)
\, \dd \vec x \nonumber\\
&+\alpha \int_\O\int_\O \r(\vec x {-} \vec x') \left(\mc{E}(\vec u, \vec P)(\vec
x, \vec x') - \frac{1}{n} \mf{E}_\r(\vec u, \vec P)(\vec x)\right) \left(\mc{E}(\vec v, \vec Q)(\vec
x, \vec x') - \frac{1}{n} \mf{E}_\r(\vec v ,\vec Q)(\vec x)\right) \,
\dd \vec x' \, \dd \vec x\nonumber\\
& +
\gamma \int_A \vec P (\vec x) : \vec Q (\vec x) \, \dd \vec x. 
\end{align*}
Making use of $B_\r$ one can equivalently rewrite
\eqref{var1}--\eqref{var2} as the nonlocal system
\begin{align*}
&2B_\r((\vec u(t), \vec P(t)), (\vec v, \vec 0)) = \into \vec b(\vec x,t)\cdot
  \vec v(\vec x) \, \dd \vec x \quad \forall \vec v \in
  \mc{S}_\r, \\
&  2B_\r((\vec u(t), \vec P(t)), (\vec 0, \dot{\vec P}(t)-\vec w) )
  \leq \into  \sigma_y |\vec w(\vec x)|\, \dd \vec x  - \into  \sigma_y |\dot{\vec P}(\vec x,t)| \, \dd
  \vec x \quad \forall \vec w \in L^2(\Omega; \Rnnsd).
\end{align*}

The quasistatic
elastoplastic evolution problem consists in finding a strong (in time)
solution to system \eqref{var1}--\eqref{var2}, starting from the initial
state $ (\ove{\vec u},\ove{\vec P})\in
(V \cap \mc{S}_{\r}) \times L^2(\O, \Rnnsd)$. We equivalently
reformulate the problem in energetic terms as that of finding {\it quasistatic evolution} trajectories $(\vec u_{\r}, \vec
p_{\r}): [0,T] \to Q  $ such
that, for all  $t \in [0,T]$,
\begin{align}
   &\vec u_\r(t) \in \mc{S}_\r \ \ \text{and} \ \ 
F_{\r}(\vec u_{\r}(t), \vec P_{\r}(t),t) \leq F_{\r}(\haz{\vec u}, \haz{\vec 
   p},t) + H(\haz{\vec P}{-}\vec P_{\r}(t)) \quad \forall (\haz{\vec u}, \haz{\vec P})
   \in Q,\label{eq:stability}\\
&F_{\r}(\vec u_{\r}(t), \vec P_{\r}(t),t) + {\rm Diss}_{[0,t]}(\vec P_\r)
= F_{\r}( \vec u_{\r}(0), \vec P_{\r}(0),0) -\int_0^t \int_\O \dot{\vec b}(\vec
x,s)\cdot \vec u_\r(\vec x,s)\, \dd \vec x\, \dd s\label{eq:energy_balance}
\end{align}
where the dissipation ${\rm Diss}_{[0,t]}(\vec P_\r)$ is defined as 
$${\rm Diss}_{[0,t]}(\vec P_\r)=\sup\left\{ \sum_{i=1}^N H(\vec
  P_{\r}(t_{i-1}){-} \vec P_{\r}(t_i)) \right\}$$
and the supremum is taken on all partitions $\{0=t_0<t_1< \dots <t_N
=t\}$ of $[0,t]$. The time-parametrized variational inequality
\eqref{eq:stability} is usually called {\it global stability}. It
expresses a minimality of the current state $(\vec u_{\r}(t), \vec
P_{\r}(t))$ with respect to possible competitors $(\haz{\vec u},
\haz{\vec P})$ when the combined effect of energy and dissipation is
taken into account. We will call all states $(\vec u_{\r}(t), \vec
P_{\r}(t))$  fulfilling \eqref{eq:stability} {\it stable} and
equivalently indicate \eqref{eq:stability} as $(\vec u_{\r}(t), \vec
P_{\r}(t)) \in \mf{S}_{\r}(t)$, so that $\mf{S}_{\r}(t)$ is the set of
{\it stable states} at time $t$.
The scalar relation \eqref{eq:energy_balance} is nothing but the
energy balance: The sum of the actual and the dissipated energy
(left-hand side of \eqref{eq:energy_balance}) equals the sum of the
initial energy and the work done by external actions (right-hand side). 
Note that systems \eqref{var1}--\eqref{var2} and
\eqref{eq:stability}--\eqref{eq:energy_balance} are equivalent as the
energy $F_\r$ is strictly convex (see Proposition \ref{prop:strict}).

This section is devoted to the study of the quasistatic evolution
problem \eqref{eq:stability}--\eqref{eq:energy_balance}.  
In particular,
we prove that it is well posed in Subsection
\ref{se:well_posedness_quasistatic} by passing to the limit into a
time-discretization discussed in Subsection
\ref{se:time_discrete}. Eventually, we study the localization limit as $\r$
converges to a Dirac delta  function at  $\vec 0$ in Subsection \ref{se:local_limit_quasi}

\subsection{Incremental minimization}\label{se:time_discrete}

For the sake of notational simplicity, we drop the subscript $\r$ from
$(\vec u_{\r}, \vec P_{\r})$ in this subsection.
Let a partition $\{0=t_0<t_1< \dots <t_N
=T\}$ of $[0,T]$ be given and let $(\vec u_0,\vec P_0) = (\ove{\vec u},\ove{\vec P})$.  The {\it
  incremental minimization problem} consists in finding $(\vec
u_i,\vec P_i) \in Q$
that minimizes 
\begin{equation}
  \label{eq:incremental_i}
F_\r(\vec u, \vec P, t_i) + H(\vec P{-} \vec P_{i-1})  
\end{equation}
for $i=1, \dots, N$. Owing to Theorem \ref{th:existence}, the unique solution $\{(\vec
u_i,\vec P_i)\}_{i=0}^N$ can be found inductively on $i$. The minimality
in \eqref{eq:incremental_i} and the triangle inequality entail that 
\begin{equation}
F_\r(\vec u_i, \vec P_i, t_i) + H(\vec P_i{-} \vec P_{i-1}) \leq
F_\r(\haz{\vec u}, \haz{\vec P}, t_i) + H( \haz{\vec P}{-} \vec
P_{i})+H(\vec P_i{-} \vec P_{i-1})\quad  \forall (\haz{\vec u}, \haz{\vec P})
   \in Q.\label{eq:stab}
\end{equation}
This proves in particular that $(\vec
u_i,\vec P_i)$ is stable for all $i$. More precisely, $(\vec
u_i,\vec P_i) \in \mf{S}_{\r}(t_i)$ for all $i=1, \dots,N$. Again from
minimality one has
\begin{equation}\label{eq:Gron1}
F_\r(\vec u_i, \vec P_i, t_i) + H(\vec P_i{-} \vec P_{i-1}) \leq
F_\r(\vec u_{i-1}, \vec P_{i-1}, t_i)  
=F_\r(\vec u_{i-1}, \vec P_{i-1}, t_{i-1}) - \int_\O \int_{t_{i-1}}^{t_i} \dot{\vec b}(\vec x,s)\, \dd s \cdot \vec u_{i-1}(\vec x) \, \dd \vec x. 
\end{equation}
Now, the coercivity of $F_\r$ from Lemma \ref{le:Fcoerc} implies the existence of $M>0$ such that
\[
 \|(\vec u, \vec P)\|_{\mc{T}_{\r}} \leq M \left( 1 + F_{\r} (\vec u, \vec P) \right) , \qquad \forall (\vec u, \vec P) \in Q .
\]
This and  Minkowski's inequality imply
\begin{equation}\label{eq:Gron2}
\begin{split}
 \int_\O \int_{t_{i-1}}^{t_i} \dot{\vec b}(\vec x,s)\, \dd s \cdot \vec u_{i-1}(\vec x) \, \dd \vec x & \leq  \Bigl\| \int_{t_{i-1}}^{t_i} \dot{\vec b}(\cdot,s)\, \dd s \Bigr\|_{L^2 (\O)} \left\| \vec u_{i-1} \right\|_{L^2 (\O)} \\
 & \leq M \int_{t_{i-1}}^{t_i} \bigl\| \dot{\vec b}(\cdot,s) \bigr\|_{L^2 (\O)} \, \dd s \left( 1 + F_{\r} (\vec u_{i-1}, \vec P_{i-1}, t_{i-1}) \right) .
 \end{split}
\end{equation}
Fix an integer $m\leq N$; by summing \eqref{eq:Gron1} up for $i=1,\dots,m$ we get
\begin{equation}\label{eq:uppertrue}
\begin{split}
F_\r(\vec u_m, \vec P_m, t_m) + \sum_{i=1}^m H(\vec P_i{-} \vec P_{i-1}) \leq &
\, F_\r(\ove{\vec u}, \ove{\vec P},0) - \sum_{i=1}^m \int_\O \int_{t_{i-1}}^{t_i} \dot{\vec b}(\vec x,s)\, \dd s \cdot \vec u_{i-1}(\vec x) \, \dd \vec x ,
 \end{split}
\end{equation}
while using \eqref{eq:Gron2} we get 
\begin{equation*}
\begin{split}
F_\r(\vec u_m, \vec P_m, t_m) + \sum_{i=1}^mH(\vec P_i{-} \vec P_{i-1}) \leq &
\, F_\r(\ove{\vec u}, \ove{\vec P},0) + M \, \| \dot{\vec b} \|_{L^1(0,T;L^2(\O;\Rn))} \\
& +M \sum_{i=1}^m \int_{t_{i-1}}^{t_i} \bigl\| \dot{\vec b}(\cdot,s) \bigr\|_{L^2 (\O)} \, \dd s \, F_{\r} (\vec u_{i-1}, \vec P_{i-1}, t_{i-1}) .
 \end{split}
\end{equation*}
With the discrete Gronwall inequality we deduce that 
\begin{equation}
  \label{eq:bound}
  F_\r(\vec u_m, \vec P_m, t_m) + \sum_{i=1}^m H(\vec P_i{-}\vec
  P_{i-1}) \leq C
\end{equation}
where $C$ depends on $F_\r(\ove{\vec u}, \ove{\vec P},0) $ and $\|
\dot{\vec b} \|_{L^1(0,T;L^2(\O;\Rn))}$ but not on the time
partition. In particular, the incremental minimization problem
delivers a stable approximation scheme. This could additionally be
combined with a space discretization as well.

\subsection{Well-posedness of the quasistatic evolution
  problem}\label{se:well_posedness_quasistatic}

The aim of this subsection is to check the following well-posedness result.

\begin{theorem}[Well-posedness of the quasistatic evolution
  problem] \label{thm1}
Let 
$\vec b \in W^{1,1}(0,T;L^2(\O;\Rn))$ and
$(\ove{\vec u}, \ove{\vec P}) \in \mf{S}_\r(0)$.  
Then there exists a unique quasistatic evolution $t \mapsto (\vec
u_\r(t), \vec P_\r (t))$.
\end{theorem}

\begin{proof}
  This well-posedness argument is quite standard, for the energy
  $F_\r$ is quadratic and coercive. Indeed, the statement follows from
  \cite[Thm.\ 3.5.2]{MielkeRoubicek15} where one finds
  quasistatic evolutions by passing to the limit in the time-discrete
  solution of the incremental problem \eqref{eq:incremental_i} as the
  fineness of the partition goes to $0$. Assume for simplicity such
  partitions to be uniform and given by $t_i^N= iT/N$ (non-uniform
  partitions can be considered as well) and define $(\vec u_N, \vec
  P_N):[0,T] \to Q$ to be the backward-in-time piecewise constant interpolant of the
  solution of the incremental problem \eqref{eq:incremental_i} on the
  partition.

Bound \eqref{eq:bound} and the coercivity of $F_\r$
  from Lemma \ref{le:Fcoerc} entail that 
$ \| (\vec u_N, \vec P_N)\|_{\mc
{T}_\r}$ and $ {\rm Diss}_{[0,T]}(\vec
P_N) $ are bounded independently of $N$.
This allows for the application of the Helly Selection Principle
\cite[Thm.\ 2.1.24]{MielkeRoubicek15} which, in
combination with Lemma \ref{le:Fdcoercivity} and Proposition \ref{prop:compactness}, entails that 
$(\vec u_N, \vec
  P_N)$ converges to $ (\vec u, \vec P)$ with respect to the
  strong $\times$ weak topology of  $Q$, for all times. 

The global
  stability $(\vec u(t), \vec P(t)) \in \mf{S}_\r(t)$ for all $t \in
  [0,T]$ follows by passing to the $\limsup$ in
  \eqref{eq:stab} by means of the so-called {\it quadratic trick}, see
  \cite[Lem.\ 3.5.3]{MielkeRoubicek15}: let $(\haz{\vec u},\haz{\vec
    P}) \in Q$ be given and define $(\haz{\vec u}_N,\haz{\vec
    P}_N) = (\vec u_N(t_i^N) + \haz{\vec u} - \vec u(t),\vec P_N(t_i^N) + \haz{\vec
    P}-\vec P(t))$.  By using  the short-hand notation $B_\r
  (\vec u, \vec P)$ for $B_\r((\vec u, \vec P), (\vec u, \vec P))$,
   from  the fact that $(\vec u_N(t), \vec P_N(t)) \in
  \mf{S}_{\r} (t_i^N)$ for $t \in (t_{i-1}^N,t_i^N]$ we deduce that 
  \begin{align}
    0 &\leq F_\r(\haz{\vec u}_N,\haz{\vec
    P}_N, t_i^N) - F_{\r}(\vec u_N(t), \vec P_N(t), t_i^N) + H(\haz{\vec
    P}_N{-} \vec P_N(t))\nonumber\\
&= B_\r (\haz{\vec u} - \vec u(t), \haz{\vec P} - \vec P(t)) +   
2B_\r\Big((\vec u_N(t),\vec P_N(t)) ,  (\haz{\vec u} {-} \vec u(t),
\haz{\vec P} {-} \vec P(t))\Big) \nonumber\\
&-\int_\O \vec b(\vec x, t_i^N) \cdot (\haz{\vec u}(\vec x) - \vec u(\vec x, t))
\, \dd \vec x + H(\haz{\vec
    P}{-} \vec P(t)).  \label{eq:thisgoes}
  \end{align} 
 Take now   the limit for $N\to \infty$ in \eqref{eq:thisgoes}
 and obtain 
\begin{align*}
  0 &\leq  B_\r (\haz{\vec u} - \vec u(t), \haz{\vec P} - \vec P(t)) +   
2B_\r\Big((\vec u(t),\vec P(t)) ,  (\haz{\vec u} - \vec u(t),
\haz{\vec P} - \vec P(t))\Big) \nonumber\\
&-\int_\O \vec b(\vec x, t) \cdot (\haz{\vec u}(\vec x) - \vec u(\vec x, t))
\, \dd \vec x + H(\haz{\vec
    P}{-} \vec P(t))\nonumber\\
&=  F_\r (\haz{\vec u}, \haz{\vec P}, t) -
F_\r ({\vec u} (t),\vec P(t), t) + H(\haz{\vec
    P}{-} \vec P(t)).
\end{align*} Since the latter holds for all $(\haz{\vec
  u},\haz{\vec P}) \in Q$, we have proved that  
$(\vec u(t), \vec P(t)) \in \mf{S}_\r(t)$. 

Inequality `$\leq$' in \eqref{eq:energy_balance} follows by passing
  to the $\liminf$ as $N \to \infty$ in \eqref{eq:uppertrue}. The opposite
  inequality is a consequence of the already checked global stability,
  see \cite[Prop.\ 2.1.23]{MielkeRoubicek15}.
Eventually, uniqueness is a consequence of the strict convexity of $F_\r$.
\end{proof}

\subsection{Localization limit}\label{se:local_limit_quasi}

The aim of this subsection is to investigate the localization limit for $\r$
converging to a Dirac delta function at $\vec 0$. Replace $\r$ by $\r_\d$ fulfilling assumptions \eqref{newR}--\eqref{newR2} of
Subsection \ref{se:Glimit} and use $\d$ as subscript instead of $\r$
wherever relevant. Define
$$\mc{S}_0 =\{ \vec u \in H^1(\O,\Rn) \ : \ \vec u =\vec 0 \ \text{on}
\ \omega\}.$$
We shall check that the
quasistatic evolution $(\vec u_\d,\vec P_\d)$ for the nonlocal model converges to the unique
solution $(\vec u_0,\vec P_0)$ of the classical local elastoplastic
quasistatic problem
\begin{align}
  \partial_{\vec u} F_0 (\vec u_0(t), \vec P_0(t),t) &= \vec 0 \ \  
  \ \text{in} \ \ \mc{S}_0^*, \label{var10}\\
\partial_{\dot{\vec P}} H(\dot{\vec P}_0(t)) + \partial_{\vec P} F_0
  (\vec u_0(t), \vec P_0(t),t) &\ni \vec 0  \ \  
  \ \text{in} \ \ L^2(\Omega;\Rnnsd).  \label{var20}
\end{align} 
In analogy with \eqref{var1}--\eqref{var2}, one can rewrite
\eqref{var10}--\eqref{var20} via the bilinear form $B_0$ 
\begin{align*}
 &B_0 \big((\vec u, \vec P), (\vec v, \vec Q)\big) =  \b \int_{\O} \div \vec
                                              u (\vec x) \div \vec v
                                              (\vec x)  \, \dd \vec x
  \\
&+ \a  n \int_{\O} \dashint_{\Sn} \left( (\nabla \vec u (\vec x) - \vec P (\vec x)) \vec z \cdot \vec z - \frac{1}{n} \div \vec u (\vec x) \right) \left( (\nabla \vec v (\vec x) - \vec Q (\vec x)) \vec z \cdot \vec z - \frac{1}{n} \div \vec v (\vec x) \right) \, \dd \Hn (\vec z) \, \dd \vec x \\
 & +
 \g \int_{\O} \vec P (\vec x) : \vec Q(\vec x) \, \dd \vec x
\end{align*}
as
\begin{align}
&2B_0((\vec u_0(t), \vec P_0(t)), (\vec v, \vec 0)) = \into \vec b(\vec
  x,t)\cdot
  \vec v(\vec x) \, \dd \vec x \quad \forall \vec v \in
  \mc{S}_0, \label{var110}\\
& 2B_0((\vec u_0(t), \vec P_0(t)), (\vec 0, \dot{\vec P}_0(t) - \vec w) )
  \leq \into  \sigma_y |\vec w(\vec x)|\, \dd \vec x  - \into  \sigma_y |\dot{\vec P}_0(\vec x,t)| \, \dd
  \vec x \quad \forall \vec w \in L^2(\Omega; \Rnnsd).
\label{var210}
\end{align}
By recalling the expression for the Lam\'e coefficients
\eqref{eq:lame} the latter can be equivalently restated in the
classical form 
\begin{align}
  &\into \boldsymbol \Sigma(\vec x,t):\nabla^s \vec v  (\vec x)\, \dd \vec x
    =  \into\vec b(\vec x,t)\cdot \vec v (\vec x)  \quad \forall \vec
    v \in V, \ \ \text{for a.e.}  \ t \in
  (0,T), \label{eq:equilibrium}\\
& \vec u(t)=\vec 0 \quad  \text{on $\partial \Omega \setminus
  \overline \o$, \ for a.e.}  \ t \in
  (0,T), \label{eq:boundary_condition}\\[2mm]
&\boldsymbol \Sigma = \lambda \,{\rm tr} \left( \nabla^s \vec u-
\vec P\right) + 2\mu \left( \nabla^s \vec u -
\vec P\right)\ \ \text{a.e. in} \  \Omega \times
  (0,T),  \label{eq:constitutive relation}\\[2mm]
& \sigma_y  \partial |\dot{\vec P}| + 2\gamma \vec P \ni \boldsymbol \Sigma\ \
  \text{a.e. in} \  \Omega \times
  (0,T),  \label{eq:flow_rule}\\
&\vec P(0)=\vec P_0\ \ \text{a.e. in} \  \Omega.  \label{eq:initial_condition}
\end{align}
Relations \eqref{var10} or \eqref{var110} correspond to  
the quasistatic equilibrium system  \eqref{eq:equilibrium} and the corresponding boundary
condition \eqref{eq:boundary_condition}. Note that, since $\O\setminus
\overline{\omega}$ is Lipschitz, condition \eqref{eq:boundary_condition}
can be also read as $\vec u(t)|_{\omega} \in H^1_0(\omega,\Rn)$. The isotropic material response is encoded by the
constitutive relation \eqref{eq:constitutive relation} for the
stress $\boldsymbol \Sigma$ (note however
that isotropy is here assumed for the sake of definiteness only, for
the analysis covers anisotropic cases with no change). The plastic
flow rule \eqref{var20} or \eqref{var210} corresponds to
\eqref{eq:flow_rule}, to be considered together with the
initial condition \eqref{eq:initial_condition}. Recall that problem
\eqref{eq:equilibrium}--\eqref{eq:initial_condition} (equivalently
systems \eqref{var10}--\eqref{var20} or \eqref{var110}--\eqref{var210}
along with initial conditions) admits a unique
strong solution in time \cite{Han-Reddy}, which is indeed a
quasistatic evolution in the sense of
\eqref{eq:stability}--\eqref{eq:energy_balance} \cite[Sec.\ 4.3.1]{MielkeRoubicek15}.

\begin{theorem}[Convergence of quasistatic evolutions] \label{thm:conv_quasi}
Let $\vec b \in W^{1,1}(0,T;L^2(\O;\Rn))$ and $(\ove{\vec u}_\d, \ove{\vec P}_\d) \in \mf{S}_\d(0)$ be such that $(\ove{\vec u}_\d, \ove{\vec P}_\d)\to (\ove{\vec u}_0, \ove{\vec
  P}_0)$ with respect to the strong $\times$ weak topology of $Q$ and $F_\d (\ove{\vec u}_\d, \ove{\vec
  P}_\d,0) \to  F_0 (\ove{\vec u}_0, \ove{\vec
  P}_0,0)$. Then, the unique quasistatic evolution of the
nonlocal problem $(\vec u_\d, \vec P_\d)$ converges to $ (\vec u_0,
\vec P_0)$ with respect to the strong $\times$  weak topology of $Q$,
for all times, where $ (\vec u_0, \vec P_0)$ is the unique quasistatic
evolution of local elastoplasticity.
\end{theorem}

\begin{proof}
  This argument follows along the general lines of
  \cite[Thm.\ 3.8]{MRS} and hinges on identifying a suitable {\it
    mutual recovery sequence} for the functionals
  $F_\r$ and $H$.  

The energy balance \eqref{eq:energy_balance} at level $\r$, the
uniform coercivity of $F_\r$ from Lemma \ref{le:Fdcoercivity}, and the
fact that 
$\dot{\vec b} \in L^1(0,T;L^2(\O,\Rn))$ entail that $ \sup_{t\in [0,T]}\| (\vec u_\d,
\vec P_d)\|_{\mc{T}_{\d}}$ and ${\rm Diss}_{[0,T]}(\vec P_\d)$ are
bounded independently of $\d$. By using the generalized Helly Selection Principle
\cite[Thm.\ A.1]{MRS}, Lemma \ref{le:Fdcoercivity},  and Proposition \ref{prop:compactness} one extracts a (non-relabeled)
subsequence converging to $(\vec u_0,\vec P_0)$ strongly $\times$ weakly in $Q$ for all times. By passing to the $\liminf$ as $\d
\to 0$ in the energy balance
\eqref{eq:energy_balance}, as $F_\d \to F_0$ in the $\Gamma$-convergence
sense (Theorem \ref{thm:gamma}) one finds that 
\begin{equation}
F_{0}(\vec u_{0}(t), \vec P_{0}(t),t) + {\rm Diss}_{[0,t]}(\vec P_0)
\leq F_{0}( \vec u_{0}(0), \vec P_{0}(0),0) -\int_0^t \int_\O \dot{\vec b}(\vec
x,s)\cdot \vec u_0(\vec x,s)\, \dd \vec x\, \dd s , \label{eq:energy_up}
\end{equation}
which is the upper energy estimate. Moreover, the initial values of $(\vec u_0,\vec P_0)$ can be computed as
$$ (\vec u_{0}(0), \vec P_{0}(0)) = \lim_{\d \to 0}  (\vec
u_{\d}(0), \vec P_{\d}(0)) =  \lim_{\d \to 0} (\ove{\vec u}_\d, \ove{\vec P}_\d)= (\ove{\vec u}_0, \ove{\vec
  P}_0) , $$
where the limit is strong $\times$ weak in $Q$.

We now need to check that $(\vec u_0,\vec
P_0)$ is globally stable for all times, namely $(\vec u_0(t),\vec
P_0(t))\in \mf{S}_0(t)$ for all $t \in [0,T]$, where the latter set of
stable states is defined starting from the
energy $F_0$. This is obtained by exploiting once again the quadratic nature of the
energy via the {\it quadratic trick}.
As $(\vec u_\d(t),\vec
P_\d(t))\in \mf{S}_\d(t)$ for all $t \in [0,T]$, for any $(\haz{\vec
  u}_\d,\haz{\vec P}_\d) \in Q $ one has that 
\begin{align}
  & 0 \leq F_\d (\haz{\vec
  u}_\d,\haz{\vec P}_\d,t) - F_\d(\vec u_\d(t),\vec P_\d(t),t) + H(
\haz{\vec P}_\d{-}\vec P_\d(t))\nonumber\\
& = B_\d (\haz{\vec   u}_\d,\haz{\vec P}_\d) - B_\d (\vec u_\d(t),\vec P_\d(t)) -
\int_\Omega \vec b (\vec x,t) \cdot (\haz{\vec
  u}_\d - \vec u_\d(t))\, \dd \vec x + H(
\haz{\vec P}_\d {-}\vec P_\d(t)). \label{eq:insert}
\end{align}
Let the competitors $(\haz{\vec u}_0,\haz{\vec P}_0) \in Q$
be given and assume for the time being that $(\haz{\vec u}_0 -\vec u_0(t) ,
\haz{\vec P}_0 -\vec P_0(t)) \in C^\infty (\bar{\Omega};\Rn \times \Rnnsd)$. Insert the {\it mutual recovery sequence}
 $$(\haz{\vec u}_\d,\haz{\vec P}_\d) = \big( \vec u_\d(t)+\haz{\vec u}_0 -\vec u_0(t) , \vec P_\d(t)+\haz{\vec P}_0 -\vec P_0(t)\big) $$
into \eqref{eq:insert} getting 
\begin{align}
  0&\leq  B_\d\big(\haz{\vec u}_0 -\vec u_0(t) ,  \haz{\vec P}_0 -\vec
  P_0(t)\big) - \int_\Omega \vec b(\vec x,t) \cdot (\haz{\vec
    u}_0(\vec x) -\vec u_0(\vec x,t))\, \dd \vec x\nonumber \\
& + H(\haz{\vec P}_0 {-}\vec
  P_0(t)) + 2B_\d\big((\vec u_\d(t),\vec P_\d(t)) , \big(\haz{\vec u}_0 {-}\vec u_0(t) ,  \haz{\vec P}_0 {-}\vec
  P_0(t)\big)\big).\label{eq:pass}
\end{align}
We aim now at passing to the  limit as $\d\to 0$
in \eqref{eq:pass}. The first two terms in the
right-hand side converge by Proposition
\ref{prop:limitFdup} and the dissipation term is independent of
$\d$. One can hence use Lemma \ref{lemmone} for the last term and conclude that 
\begin{align}
  0&\leq  B_0\big(\haz{\vec u}_0 -\vec u_0(t) ,  \haz{\vec P}_0 -\vec
  P_0(t)\big) - \int_\Omega \vec b(\vec x,t) \cdot (\haz{\vec
    u}_0(\vec x) -\vec u_0(\vec x,t))\, \dd \vec x\nonumber \\
& + H(\haz{\vec P}_0 {-}\vec
  P_0(t)) + 2B_0\big((\vec u_0(t),\vec P_0(t)) , \big(\haz{\vec u}_0 -\vec u_0(t) ,  \haz{\vec P}_0 -\vec
  P_0(t)\big)\big)\nonumber\\
&=  F_0 (\haz{\vec
  u}_0,\haz{\vec P}_0,t) - F_0 (\vec u_0(t),\vec P_0(t),t) + H(
\haz{\vec P}_0 {-} \vec P_0(t)).\nonumber
\end{align}
The stability of $(\vec u_0(t),\vec P_0(t))$ is hence checked against
all competitors with $(\haz{\vec u}_0 -\vec u_0(t) ,  \haz{\vec P}_0 -\vec
  P_0(t)) $ in $ C^\infty (\bar{\Omega};\Rn \times \Rnnsd)$. In
  order to conclude for the global stability of  $(\vec u_0(t),\vec
Q_0(t))$ at time $t$ one has now to argue by approximation. Let a
general competitor $(\haz{\vec u}_0 , \haz{\vec P}_0)\in Q$ with $\haz{\vec u}_0 \in H^1 (\O, \Rn)$ be given
and choose a sequence $(\haz{\vec u}_{0j} , \haz{\vec P}_{0j})\in Q$ such that $(\haz{\vec u}_{0j} ,
\haz{\vec P}_{0j}) \to (\haz{\vec u}_0 , \haz{\vec P}_0)$ strongly in
$H^1 (\O, \Rn) \times L^2 (\O, \Rnnsd)$ and $(\haz{\vec u}_{0j} -\vec u_0(t) ,  \haz{\vec P}_{0j} -\vec
  P_0(t)) \in C^\infty (\bar{\Omega};\Rn \times \Rnnsd)$. As $F_0$ and $H$ are continuous with respect to the strong  
topology in $H^1 (\O, \Rn) \times L^2 (\O, \Rnnsd)$ and $L^2(\O,\Rnnsd)^2$, respectively,  one gets
\begin{align*}
  0&\leq \lim_{j \to \infty}\Big( F_0 (\haz{\vec
  u}_{0j},\haz{\vec P}_{0j},t) - F_0(\vec u_0(t),\vec P_0(t),t) + H(
\haz{\vec P}_{0j} {-} \vec P_0(t))\Big)\\
&=  F_0 (\haz{\vec
  u}_0,\haz{\vec P}_0,t) - F_0(\vec u_0(t),\vec P_0(t),t) + H(
\haz{\vec P}_0 {-} \vec P_0(t))
\end{align*}
which proves $(\vec u_0(t),\vec
Q_0(t))\in \mf{S}_0(t)$. Eventually,  global stability allows to
recover the
opposite estimate to \eqref{eq:energy_up} as in
\cite[Prop.\ 2.1.23]{MielkeRoubicek15}.

We have hence proved that $(\vec u_0,\vec P_0)$ is a quasistatic evolution of the local elastoplastic problem. As $F_0$ is strictly convex, such solution is unique and convergence holds for the whole sequence.
\end{proof}

\section*{Acknowledgements}

C.M.-C.\ has been supported by the Spanish Ministry of Economy and Competitivity (Project MTM2014-57769-C3-1-P and the ``Ram\'on y Cajal'' programme RYC-2010-06125) and the ERC Starting grant no.\ 307179.
U.S.\ acknowledges the support by the Vienna Science and Technology Fund (WWTF)
through Project MA14-009 and by the Austrian Science Fund (FWF)
projects F\,65  and  P\,27052. M.K. and U.S.\ acknowledge the support by the
FWF-GA\v{C}R project    I\,2375-16{-}34894L  and by the 
OeAD-M\v{S}MT  project   CZ~17/2016-7AMB16AT015.

\appendix

\section{Auxiliary results}\label{se:auxiliary}

 We collect here some auxiliary results that have been used in
the paper. 
Let $\f \in C^{\infty}_c (\Rn)$ satisfy $\supp \f \subset B (\vec 0, 1)$, $\f \geq 0$, and $\int_{\Rn} \f \, \dd \vec x = 1$.
For each $r>0$, define the function $\f_r \in C^{\infty}_c (\Rn)$ as $\f_r (\vec x) = r^{-n}  \f (\vec x / r)$.
Define $\O_r = \{ \vec x \in \O : \dist (\vec x, \p \O) > r \}$.
As usual, given a function $u : \O \to \R$ its mollification $\f_r \star u : \O_r \to \R$ is defined as
\[
 (\f_r \star u) (\vec x) = \int_{B (\vec 0, r)} \f_r (\vec z) \, u (\vec x - \vec z) \, \dd \vec z .
\]
For vector-valued functions, the mollification is defined componentwise.

The following result was used in Section \ref{se:Glimit}.

\begin{lemma}[Energy decreases by mollification]\label{le:Jensen}
Let $(\vec u, \vec P) \in \mc{T}_{\r} (\O)$.
Let $A \ssubset \O$ be measurable and let $0 < r < \dist (A, \p \O)$.
Then
\begin{align*}
 & \int_A \int_A \r (\vec x {-} \vec x') \left( \mc{E} (\f_r \star \vec u, \f_r \star \vec P) (\vec x, \vec x') - \frac{1}{n} \f_r \star \mf{E}_{\r} (\vec u, \vec P) (\vec x) \right)^2 \dd \vec x' \, \dd \vec x \\
 \leq & \int_{\O} \int_{\O} \r (\vec x {-} \vec x') \left( \mc{E} (\vec u, \vec P) (\vec x, \vec x') - \frac{1}{n} \mf{E}_{\r} (\vec u, \vec P) (\vec x) \right)^2 \dd \vec x' \, \dd \vec x .
\end{align*}
\end{lemma}
\begin{proof}
For each $\vec x, \vec x' \in A$,
\[
 \mc{E} (\f_r \star \vec u, \f_r \star \vec P) (\vec x, \vec x') - \frac{1}{n} \f_r \star \mf{E}_{\r} (\vec u, \vec P) (\vec x) = \int_{B (\vec 0, r)} \f_r (\vec z) \left( \mc{E} (\vec u, \vec P) (\vec x - \vec z, \vec x' {-} \vec z) - \frac{1}{n} \mf{E}_{\r} (\vec u, \vec P) (\vec x - \vec z) \right) \dd \vec z ,
\]
so, by Jensen's inequality,
\begin{align*}
 & \left( \mc{E} (\f_r \star \vec u, \f_r \star \vec P) (\vec x, \vec x') - \frac{1}{n} \f_r \star \mf{E}_{\r} (\vec u, \vec P) (\vec x) \right)^2 \\
 \leq & \int_{B (\vec 0, r)} \f_r (\vec z) \left( \mc{E} (\vec u, \vec P) (\vec x - \vec z, \vec x' {-} \vec z) - \frac{1}{n} \mf{E}_{\r} (\vec u, \vec P) (\vec x - \vec z) \right)^2 \dd \vec z .
\end{align*}
Therefore,
\begin{align*}
 & \int_A \int_A \r (\vec x {-} \vec x') \left( \mc{E} (\f_r \star \vec u, \f_r \star \vec P) (\vec x, \vec x') - \frac{1}{n} \f_r \star \mf{E}_{\r} (\vec u, \vec P) (\vec x) \right)^2 \dd \vec x' \, \dd \vec x \\
 \leq & \int_{B (\vec 0, r)} \f_r (\vec z) \int_{\O_r} \int_{\O_r} \r (\vec x {-} \vec x') \left( \mc{E} (\vec u, \vec P) (\vec x - \vec z, \vec x' {-} \vec z) - \frac{1}{n} \mf{E}_{\r} (\vec u, \vec P) (\vec x - \vec z) \right)^2  \dd \vec x' \, \dd \vec x \, \dd \vec z .
\end{align*}
But, for each $\vec z \in B (\vec 0, r)$,
\begin{align*}
 & \int_A \int_A \r (\vec x {-} \vec x') \left( \mc{E} (\vec u, \vec P) (\vec x - \vec z, \vec x' {-} \vec z) - \frac{1}{n} \mf{E}_{\r} (\vec u, \vec P) (\vec x - \vec z) \right)^2  \dd \vec x' \, \dd \vec x \\
 = & \int_{A - \vec z} \int_{A - \vec z} \r (\vec x {-} \vec x') \left( \mc{E} (\vec u, \vec P) (\vec x, \vec x') - \frac{1}{n} \mf{E}_{\r} (\vec u, \vec P) (\vec x) \right)^2  \dd \vec x' \, \dd \vec x \\
 \leq & \int_{\O} \int_{\O} \r (\vec x {-} \vec x') \left( \mc{E} (\vec u, \vec P) (\vec x, \vec x') - \frac{1}{n} \mf{E}_{\r} (\vec u, \vec P) (\vec x) \right)^2  \dd \vec x' \, \dd \vec x ,
\end{align*}
so
\begin{align*}
 & \int_{B (\vec 0, r)} \f_r (\vec z) \int_A \int_A \r (\vec x {-} \vec x') \left( \mc{E} (\vec u, \vec P) (\vec x - \vec z, \vec x' {-} \vec z) - \frac{1}{n} \mf{E}_{\r} (\vec u, \vec P) (\vec x - \vec z) \right)^2  \dd \vec x' \, \dd \vec x \, \dd \vec z \\
 \leq & \int_{\O} \int_{\O} \r (\vec x {-} \vec x') \left( \mc{E} (\vec u, \vec P) (\vec x, \vec x') - \frac{1}{n} \mf{E}_{\r} (\vec u, \vec P) (\vec x) \right)^2  \dd \vec x' \, \dd \vec x
\end{align*}
and the proof is concluded.
\end{proof}

We now show an elementary calculation of some integrals in a ball, where we exploit that the kernel is radial.

\begin{lemma}[Radially symmetric kernels]\label{le:formulasB}
Let $\r \in L^1_{\loc} (\Rn)$ and let $\bar{\r} : [0, \infty) \to [0, \infty)$ be such that $\r (\vec x) = \bar{\r} (|\vec x|)$ for a.e.\ $\vec x \in \Rn$.
Let $r>0$.
The following holds:
\begin{enumerate}[a)]
\item\label{item:homa} Let $f \in L^{\infty}_{\loc} (\Rn)$ be positively homogeneous of degree $0$.
Then
\[
 \int_{B(\vec 0, r)} \r (\vec x) \, f (\vec x) \, \dd \vec x = \int_{B(\vec 0, r)} \r (\vec x) \, \dd \vec x \ \dashint_{\Sn} f (\vec z) \, \dd \Hn (\vec z) .
\]
\item\label{item:homb} Let $\vec A \in \Rnn$.
Then
\[
 \int_{B(\vec 0, r)} \r (\vec x) \frac{\vec A \vec x \cdot \vec x}{|\vec x|^2} \, \dd \vec x = \frac{1}{n} \int_{B(\vec 0, r)} \r (\vec x) \, \dd \vec x \, \tr \vec A .
\]
\end{enumerate}
\end{lemma}
\begin{proof}
We start with \emph{\ref{item:homa})}.
We use the coarea formula and the homogeneity of $f$ to find that
\[
 \int_{B(\vec 0, r)} \r (\vec x) \, f (\vec x) \, \dd \vec x = \int_0^r \bar{\r}(s) \int_{\p B (\vec 0, s)} f (\vec x) \, \dd \Hn (\vec x) \, \dd s = \int_0^r s^{n-1} \, \bar{\r}(s)  \, \dd s \int_{\Sn} f (\vec z) \, \dd \Hn (\vec z) .
\]
The above formula applied to the constant function $f=1$ shows that
\[
  \int_{B(\vec 0, r)} \r (\vec x) \, \dd \vec x = \Hn (\Sn) \int_0^r s^{n-1} \, \bar{\r}(s)  \, \dd s .
\]
Putting the two formulas together concludes the proof of \emph{\ref{item:homa})}.

For part \emph{\ref{item:homb})}, we apply \emph{\ref{item:homa})} to the function $f(\vec x) = \frac{1}{|\vec x|^2} \vec A \vec x \cdot \vec x$ and obtain that
\[
 \int_{B(\vec 0, r)} \r (\vec x) \frac{\vec A \vec x \cdot \vec x}{|\vec x|^2} \, \dd \vec x = \int_{B(\vec 0, r)} \r (\vec x) \, \dd \vec x \ \dashint_{\Sn} \vec A \vec z \cdot \vec z \, \dd \Hn (\vec z) .
\]
Now let $\vec A_s = \frac{1}{2}(\vec A + \vec A^\top)$.
Then $\vec A \vec z \cdot \vec z = \vec A_s \vec z \cdot \vec z$ for all $\vec z \in \Rn$ and $\tr \vec A = \tr \vec A_s$.
Let $\l_1, \ldots, \l_n$ be the eigenvalues of $\vec A_s$, let $\vec R \in O (n)$ and $\vec D \in \Rnn$ be such that $\vec A_s = \vec R \vec D \vec R^\top$ and $\vec D$ is diagonal with entries $\l_1, \ldots, \l_n$.
A change of variables shows that
\[
 \int_{\Sn} \vec A_s \vec z \cdot \vec z \, \dd \Hn (\vec z) = \int_{\Sn} \vec D \vec z \cdot \vec z \, \dd \Hn (\vec z) = \int_{\Sn} \sum_{i=1}^n \l_i z_i^2 \, \dd \Hn (\vec z) .
\]
Another change of variables shows that for all $i \in \{1, \ldots, n\}$,
\[
 \int_{\Sn} z_i^2 \, \dd \Hn (\vec z) = \int_{\Sn} z_1^2 \, \dd \Hn (\vec z) ,
\]
so
\[
 \Hn (\Sn) = \int_{\Sn} |\vec z|^2 \, \dd \Hn (\vec z) = n \int_{\Sn} z_1^2 \, \dd \Hn (\vec z) .
\]
Thus,
\[
 \dashint_{\Sn} \sum_{i=1}^n \l_i z_i^2 \, \dd \Hn (\vec z) = \sum_{i=1}^n \l_i \dashint_{\Sn} z_1^2 \, \dd \Hn (\vec z) = \frac{1}{n}  \sum_{i=1}^n \l_i = \frac{1}{n} \tr \vec A ,
\]
which concludes the proof.
\end{proof}

\section{Convergence lemma}\label{sec:appendix}

We present here the proof of the key convergence lemma used for
passing to the limit in \eqref{eq:pass} in the proof of Theorem \ref{thm:conv_quasi}.

\begin{lemma}[Convergence of the bilinear term]\label{lemmone}
  Let $(\vec u_\d,\vec P_\d)\to (\vec u_0, \vec P_0)$
  strongly $\times$ weakly in $Q$, $(\tilde{\vec u},\tilde{\vec P}) \in C^\infty (\bar{\Omega}; \Rn \times \Rnnsd)$, and $\|(\vec u_\d,\vec
p_\d)\|_{\mc{T}_\d}$ be bounded independently of $\d$. Then 
\begin{equation*}
  B_\d\big((\vec u_\d,\vec P_\d),(\tilde{\vec u},\tilde{\vec P})\big)
  \to B_0\big((\vec u_0,\vec P_0),(\tilde{\vec u},\tilde{\vec P})\big).
\end{equation*}
\end{lemma}

\begin{proof} We aim at computing the limit of 
  \begin{align*}
    & B_\d\big((\vec u_\d,\vec P_\d),(\tilde{\vec u},\tilde{\vec
      P})\big) = 
\gamma \int_A \vec P_\d (\vec x) : \tilde{\vec P} (\vec x) \, \dd
\vec x +
\beta \int_\O \mf{D}_\d  (\vec u_\d)(\vec x) \, \mf{D}_\d (\tilde{\vec u})(\vec x)
\, \dd \vec x \nonumber\\
&+\alpha  \int_\O\int_\O \r_\d(\vec x {-} \vec x') \left(\mc{E}(\vec u_\d, \vec P_\d)(\vec
x, \vec x') - \frac{1}{n} \mf{E}_\d(\vec u_\d, \vec P_\d)(\vec x)\right) \left(\mc{E}(\tilde{\vec u}, \tilde{\vec P})(\vec
x, \vec x') - \frac{1}{n} \mf{E}_\d(\tilde{\vec u} ,\tilde{\vec P})(\vec x)\right) \,
\dd \vec x' \, \dd \vec x.
  \end{align*}
Passing to the limit in the $\gamma$ term is straightforward as $\vec
P_\d \weakc \vec P_0$ in $L^2(\Omega;\Rnnsd)$. The
$\beta$ terms goes to the limit as well, for we have that $\mf{D}_\d
(\vec u_\d) \weakc \div \vec u_0$ in $L^2(\Omega)$ \cite[Lemma
3.6]{MeDu15} and  $\mf{D}_\d
(\tilde{\vec u}) \to \div \tilde{\vec u}$ strongly in $L^2(\Omega)$
\cite[Lemma 3.1]{MeDu15} (see also Lemma \ref{le:Eddiv}). We will hence focus on the $\alpha$
term, from which, for simplicity of notation, we omit the parameter $\alpha$:
\begin{align*}
&A_\d\big((\vec u_\d,\vec P_\d),(\tilde{\vec u},\tilde{\vec
      p})\big) \\
&=   \int_\O\int_\O \r_\d(\vec x {-} \vec x') \left(\mc{E}(\vec u_\d, \vec P_\d)(\vec
x, \vec x') - \frac{1}{n} \mf{E}_\d(\vec u_\d, \vec P_\d)(\vec x)\right) \left(\mc{E}(\tilde{\vec u}, \tilde{\vec P})(\vec
x, \vec x') - \frac{1}{n} \mf{E}_\d(\tilde{\vec u} ,\tilde{\vec P})(\vec x)\right) \,
\dd \vec x' \, \dd \vec x.
\end{align*}
The strategy of the proof is that of decomposing $A_\d$ in a sum of
integrals and discuss the corresponding limits separately. We proceed
in subsequent steps. 
\bigskip

\noindent {\bf Step 1.} Let us start by simplifying the problem of
computing the limit of $A_\d$ by
replacing $\mf{E}_\d(\vec u_\d,\vec P_\d)$ and $\mf{E}_\d(\tilde{\vec
  u},\tilde{\vec P})$ by $\div \vec u_0$ and $\div \tilde{\vec
  u}$, respectively. In particular, within this step we aim at proving that 
\begin{align}
 & \lim_{\d \to 0} \left[ A_\d\big((\vec u_\d,\vec P_\d),(\tilde{\vec u},\tilde{\vec
      P})\big) - \tilde A_\d\big((\vec u_\d,\vec P_\d),(\tilde{\vec u},\tilde{\vec
      P});\vec u_0\big) \right] \label{lim1} = 0 ,
\end{align}
where we have set 
\begin{align}
&\tilde A_\d\big((\vec u_\d,\vec P_\d),(\tilde{\vec u},\tilde{\vec
      P}) ;\vec u_0\big) \nonumber\\
&=   \int_\O\int_\O \r_\d(\vec x {-} \vec x') \left(\mc{E}(\vec u_\d, \vec P_\d)(\vec
x, \vec x') - \frac{1}{n} \div \vec u_0 (\vec x)\right) \left(\mc{E}(\tilde{\vec u}, \tilde{\vec P})(\vec
x, \vec x') - \frac{1}{n} \div \tilde{\vec u} (\vec x)\right) \,
\dd \vec x' \, \dd \vec x.\nonumber
\end{align}
In order to do so, let us write 
$$ A_\d\big((\vec u_\d,\vec P_\d),(\tilde{\vec u},\tilde{\vec
      P})\big) -\tilde A_\d\big((\vec u_\d,\vec P_\d),(\tilde{\vec u},\tilde{\vec
      P}) ;\vec u_0\big)=J^1_\d+J^2_\d$$
with 
\begin{align*}
  J^1_\d&=  - \frac{1}{n}  \int_\O\int_\O \r_\d(\vec x {-} \vec x') \left(\mc{E}(\vec u_\d, \vec P_\d)\big(\vec
x, \vec x'\big) - \frac{1}{n} \mf{E}_\d(\vec u_\d,\vec P_\d)(\vec x)\right) \left(\mf{E}_\d(\tilde{\vec
  u},\tilde{\vec P})(\vec x) -\div \tilde{\vec u} (\vec x)\right) \,
\dd \vec x' \, \dd \vec x,\nonumber\\
J^2_\d&= - \frac{1}{n}  \int_\O\int_\O \r_\d(\vec x {-} \vec x')
\big(\mf{E}_\d(\vec u_\d,\vec P_\d)(\vec x) -\div \vec u_0 (\vec x)\big) \left(\mc{E}(\tilde{\vec u}, \tilde{\vec P})(\vec
x, \vec x') - \frac{1}{n} \div \tilde{\vec u} (\vec x)\right) \,
\dd \vec x' \, \dd \vec x
\end{align*}
and prove that $J^1_\d \to 0$ and $J^2_\d\to 0$ as $\d \to 0.$

As regards $J^1_\d$, one has the bound 
\begin{align*}
 \left| J^1_\d \right| &\leq \frac{1}{n} \left(\int_\O\int_\O \r_\d(\vec x {-} \vec x') \left(\mc{E}(\vec u_\d, \vec P_\d)(\vec
x, \vec x') - \frac{1}{n} \mf{E}_\d(\vec u_\d,\vec P_\d)(\vec
x)\right)^2  \,
\dd \vec x' \, \dd \vec x\right)^{1/2} \\
&\times \left(\int_\O\int_\O \r_\d(\vec x {-} \vec x') \big(\mf{E}_\d(\tilde{\vec
  u},\tilde{\vec P})(\vec x) -\div \tilde{\vec u} (\vec x)\big)^2  \,
\dd \vec x' \, \dd \vec x\right)^{1/2}.
\end{align*}
The first integral in the right-hand side above is bounded as $\| (\vec
u_\d,\vec P_\d)\|_{\mathcal T_\d}$ is bounded whereas the second
integral tends to $0$ because of Lemma \ref{le:Eddiv}.a. 

Next, we rewrite
\begin{align*}
J^2_\d =-\frac{1}{n}  \int_\O
\big(\mf{E}_\d(\vec u_\d,\vec P_\d)(\vec x) -\div \vec u_0 (\vec x)\big) \left(\int_\O \r_\d(\vec x {-} \vec x')\left(\mc{E}(\tilde{\vec u}, \tilde{\vec P})(\vec
x, \vec x') - \frac{1}{n} \div \tilde{\vec u} (\vec x)\right) \,
\dd \vec x' \right) \dd \vec x.
\end{align*}
We have that $\mf{E}_\d(\vec
u_\d,\vec P_\d) \weakc \div \vec u_0$ in $L^2(\Omega)$ by Lemma
\ref{le:Eddiv}.b. On the other hand, by arguing as in the proof
Proposition \ref{prop:limitFdup} one gets that the function 
$$ \vec x \mapsto \int_\O \r_\d(\vec x {-} \vec x')\left(\mc{E}(\tilde{\vec u}, \tilde{\vec P})(\vec
x, \vec x') - \frac{1}{n} \div \tilde{\vec u} (\vec x)\right) \,
\dd \vec x' $$
is strongly convergent in $L^2(\Omega)$ and $J^2_\d \to 0$ follows.
\bigskip

\noindent {\bf Step 2: decomposition of $\tilde A_\d$.} Owing to
\eqref{lim1} we now argue directly on $\tilde A_\d$ by decomposing it as
\begin{equation}
\tilde A_\d\big((\vec u_\d,\vec P_\d),(\tilde{\vec u},\tilde{\vec
      P}) ;\vec u_0\big) = I^1_\d + I^2_\d + I^3_\d + I^4_\d ,
\label{decomp}
\end{equation}
where 
\begin{align*}
  I^1_\d &=  \int_\O\int_\O \r_\d(\vec x {-} \vec x')\,
  \mc{E}(\vec u_\d, \vec P_\d)(\vec
x, \vec x')  \,\mc{E}(\tilde{\vec u}, \tilde{\vec P})(\vec
x, \vec x')  \,
\dd \vec x' \, \dd \vec x,\\
 I^2_\d &= -\frac{1}{n}  \int_\O\int_\O \r_\d(\vec x {-} \vec x') \, \mc{E}(\vec u_\d, \vec P_\d)(\vec
x, \vec x') \, \div \tilde{\vec u} (\vec x)  \,
\dd \vec x' \, \dd \vec x,\\
I^3_\d&=-\frac{1}{n}  \int_\O\int_\O \r_\d(\vec x {-} \vec x') \,\div
\vec u_0 (\vec x) \, \mc{E}(\tilde{\vec u}, \tilde{\vec P})(\vec
x, \vec x')  \,
\dd \vec x' \, \dd \vec x,\\
I^4_\d&=\frac{1}{n^2} \int_\O\int_\O \r_\d(\vec x {-} \vec x') \,\div
\vec u_0 (\vec x) \,\div \tilde{\vec u}(\vec x)  \,
\dd \vec x' \, \dd \vec x.
\end{align*}
We discuss each of these integrals in the following steps.
\bigskip

\noindent {\bf Step 3: Integral $I^1_\d$.} As in \eqref{eq:DEp}, we decompose the integral  
as $I^1_\d=I^{11}_\d+I^{12}_\d+I^{13}_\d$ where
\begin{align*}
I^{11}_\d& = -  \int_\O\int_\O \r_\d(\vec x - \vec
 x') \,\frac{\vec P_\d(\vec x)(\xx)}{|\xx|^2}\cdot (\xx) \, \mc{E}(\tilde{\vec u},\tilde{\vec P}) (\vec
x, \vec x') \, \dd \vec x' \, \dd
  \vec x,\\
  I^{12}_\d&=   \int_\O\int_\O \r_\d(\vec x - \vec
 x') \,\mc{D}(\vec u_\d-\vec u_0)(\vec x,\vec
  x')\,  \mc{E}(\tilde{\vec u},\tilde{\vec P}) (\vec
x, \vec x')\, \dd \vec x' \, \dd
  \vec x,\\
I^{13}_\d &=   \int_\O\int_\O \r_\d(\vec x - \vec
 x') \,\mc{D}(\vec u_0)(\vec x,\vec
  x')\,  \mc{E}(\tilde{\vec u},\tilde{\vec P}) (\vec
x, \vec x')\, \dd \vec x' \, \dd
  \vec x,
\end{align*}
and argue on each term separately.

In order to compute the limit of $I^{11}_\d$, let us further decompose
it as
\begin{align*}
  &I^{11}_\d =  I^{111}_\d +I^{112}_\d\\
&=  -  \int_\O\int_\O \r_\d(\vec x - \vec
 x') \,\frac{\vec P_\d(\vec x)(\xx)}{|\xx|^2}\cdot (\xx) \,
 \frac{(\nabla \tilde{\vec u}(\vec x) - \tilde{\vec P} (\vec
   x))(\xx)}{|\xx|^2}\cdot (\xx) \, \dd \vec x' \, \dd
  \vec x \\
&-   \int_\O\int_\O \r_\d(\vec x - \vec
 x') \,\frac{\vec P_\d(\vec x)(\xx)}{|\xx|^2}\cdot (\xx) \,
 \frac{(\tilde{\vec u} (\vec x') - \tilde{\vec u}(\vec x) -\nabla \tilde{\vec u}(\vec x)(\xx))}{|\xx|^2}\cdot (\xx) \, \dd \vec x' \, \dd
  \vec x.
\end{align*}
The limit of $I^{111}_\d$ can be computed by observing that the
integrand is positively homogeneous of degree $0$ in $\xx$. In
particular, arguing as in Lemma \ref{le:formulasB} we can prove that 
\begin{align*}
  \lim_{\d \to 0} \left[ I^{111}_\d + n \int_\O \dashint_{\mathbb S^{n-1}}\vec P_\d(\vec x)\vec z\cdot \vec
  z   \, (\nabla \tilde{\vec u}(\vec x) - \tilde{\vec P} (\vec
   x)) \vec z \cdot \vec z \, \dd \mc{H}^{n-1}(\vec z)\, \dd
  \vec x \right] = 0
\end{align*}
and then
\begin{align*}
 & \lim_{\d \to 0} - n \int_\O \dashint_{\mathbb S^{n-1}}\vec P_\d(\vec x)\vec z\cdot \vec
  z   \, (\nabla \tilde{\vec u}(\vec x) - \tilde{\vec P} (\vec
   x)) \vec z \cdot \vec z \, \dd \mc{H}^{n-1}(\vec z)\, \dd
  \vec x \\
  & = - n \int_\O \dashint_{\mathbb S^{n-1}}\vec P (\vec x)\vec z\cdot \vec
  z   \, (\nabla \tilde{\vec u}(\vec x) - \tilde{\vec P} (\vec
   x)) \vec z \cdot \vec z \, \dd \mc{H}^{n-1}(\vec z)\, \dd
  \vec x .
\end{align*}

In order to handle the integral $I^{112}_\d$ let us firstly observe that, as in \eqref{eq:diffgrad},
\begin{equation}  \left| \frac{(\tilde{\vec u}(\vec x') - \tilde{\vec u}(\vec x) - \nabla \tilde{\vec
      u}(\vec x) (\xx))}{|\xx|^2}\cdot (\xx) \right| \leq
  \sigma(|\xx|)\label{mod}
\end{equation}
where $\sigma$ is a modulus of continuity, and that, for all  
 $A\subset\subset \Omega$, $0<r<{\rm dist} (A,\partial
\Omega)$, and  $\vec x\in A$ we have, as in \eqref{eq:C2},
\begin{align*}
\int_{\O - \vec x} \r_\d(\tilde{\vec x})\sigma (|\tilde{\vec x}|) \,
\dd \tilde{\vec x} \leq n \sigma(r) + \| \sigma\|_\infty \int_{
  \Rn \setminus B(\vec 0,r)} \r_\d(\tilde{\vec x}) \,
\dd \tilde{\vec x}.
\end{align*}
Define now the tensor-valued functions
\begin{align*}
 \vec x \mapsto  \vec G_\d(\vec x) = \int_\Omega \r_\d(\vec x - \vec
 x') \frac{(\xx) \otimes (\xx)}{|\xx|^2} \frac{(\tilde{\vec u}(\vec x') - \tilde{\vec u}(\vec x) - \nabla \tilde{\vec
      u}(\vec x) (\xx))}{|\xx|^2}\cdot (\xx) \, \dd \vec x'
\end{align*}
and control them  for a.e.\ $\vec x\in A$ as follows
\begin{align*}
  \left| \vec G_\d(\vec x) \right|
\leq \int_{\O - \vec x} \r_\d(\tilde{\vec x}) \sigma(|\tilde{\vec
  x}|)\, \dd \tilde{\vec x}
\leq  n \sigma(r) + \| \sigma\|_\infty \int_{
  \Rn \setminus B(\vec 0,r)} \r_\d(\tilde{\vec x}) \,
\dd \tilde{\vec x}.
\end{align*}
As the right-hand side goes to $0$ as $\d \to 0$, $\sigma(r)$ can be made
arbitrarily small by choosing $r\to 0$, and $A \subset \subset \Omega$
is arbitrary we have proved that  $\vec G_\d(\vec x) \to \vec 0$ a.e. The
above bound proves additionally that  $\vec G_\d$ are equiintegrable. In
particular, $\vec G_\d \to \vec 0$ strongly in $L^2(\Omega)$. As $\vec P_\d$ is bounded in $L^2(\Omega;\Rnnsd)$ one gets that
$I^{112}_\d\to 0$ as $\d \to 0$.

The treatment of integral $I^{12}_\d$ requires a nonlocal
integration-by-parts formula, see \cite[Lemma 2.9]{MeDu15}. Indeed, for all $\varphi \in
C^\infty (\bar{\Omega} \times \bar{\Omega})$ a direct computation ensures
that 
\begin{align}
\int_\O \int_\O \r_\d(\xx) \mc{D}(\vec u)(\vec x, \vec x') \,
\varphi(\vec x, \vec x')\, \dd \vec x' \, \dd \vec x=- \int_\O \vec
u(\vec x) \cdot \mf{D}_\d^* (\varphi)(\vec x)\, \dd \vec x
\label{byparts}
\end{align}
where the vector-valued operator $\mf{D}^*_\d$ is given by 
$$\mf{D}^*_\d(\varphi)(\vec x) =  {\rm p.v.}
\int_\O  \r_\d(\xx)   \frac{\varphi(\vec x, \vec x')+
\varphi(\vec x'\,\vec x)}{|\xx|^2}(\xx)\, \dd \vec x'.$$
Let us apply formula \eqref{byparts} to $I^{12}_\d$, getting
$$ I^{12}_\d = -  \int_\Omega (\vec u_\d - \vec
u_0)(\vec x) \cdot \mf{D}^*_\d(\mc{E}(\tilde{\vec u},\tilde{\vec
  P}))(\vec x) \, \dd \vec x. $$
Since
$\vec u_\d \to \vec u_0$ strongly in $L^2(\Omega;\Rn)$, in order to
check that $I^{12}_\d\to 0$ as $\d\to 0$ one needs to provide an $L^2$
bound on $\mf{D}^*_\d(\mc{E}(\tilde{\vec u},\tilde{\vec P}))$. As
$\tilde{\vec u}$ and $\tilde{\vec P}$ are smooth, this follows along
the lines of \cite[Formula (2.3)]{MeSp15}. 

Let us now turn to the analysis of integral $I^{13}_\d$. Once again,
some further decomposition is needed. We write
$I^{13}_\d=I^{131}_\d+I^{132}_\d+I^{133}_\d $ where
\begin{align*}
  I^{131}_\d&=   \int_\O\int_\O \r_\d(\xx) \, \frac{\nabla
    \vec u_0(\vec x) (\xx)\cdot (\xx)}{|\xx|^2} \, \frac{(\nabla
    \tilde{\vec u}(\vec x)-\tilde{\vec P}(\vec x)) (\xx)\cdot
    (\xx)}{|\xx|^2} \, \dd \vec x' \, \dd \vec x,\\
I^{132}_\d&=   \int_\O\int_\O \r_\d(\xx) \, \frac{\nabla
    \vec u_0(\vec x) (\xx)\cdot (\xx)}{|\xx|^2} \, \frac{(\tilde{\vec
   u}(\vec x') - \tilde{\vec u}(\vec x) - \nabla
    \tilde{\vec u}(\vec x) (\xx))\cdot (\xx)}{|\xx|^2}\, \dd \vec x' \, \dd \vec x,\\
 I^{133}_\d&=   \int_\O\int_\O \r_\d(\xx) \, \frac{(\vec
   u_0(\vec x') - \vec u_0(\vec x) - \nabla
    \vec u_0(\vec x) (\xx))\cdot (\xx)}{|\xx|^2} \\
&\qquad\times  \frac{(\tilde{\vec
   u}(\vec x') - \tilde{\vec u}(\vec x) -\tilde{\vec P}(\vec x) (\xx))\cdot
    (\xx)}{|\xx|^2} \, \dd \vec x' \, \dd \vec x.\\
\end{align*}
The integrand of $I^{131}_\d$ is positively homogeneous of degree $0$
in $\xx$. By arguing as in Lemma \ref{le:formulasB} one can prove that 
$$I^{131}_\d \to  n \int_\O \dashint_{\mathbb
    S^{n-1}} \nabla \vec u_0(\vec x)  \vec z\cdot \vec
  z   \, \big(\nabla \tilde{\vec u}(\vec x) - \tilde{\vec P} (\vec
   x)\big) \vec z \cdot \vec z \,\, \dd \mc{H}^{n-1}(\vec z)\, \dd
  \vec x\quad \text{as $\d \to 0$}.$$
Integral $I^{132}_\d$ can be proved to converge to $0$ by arguing
similarly as in $I^{111}_\d$, as (compare with \eqref{mod})
$$\left| \frac{\nabla
    \vec u_0(\vec x) (\xx)\cdot (\xx)}{|\xx|^2} \, \frac{(\tilde{\vec
   u}(\vec x') - \tilde{\vec u}(\vec x) - \nabla
    \tilde{\vec u}(\vec x) (\xx))\cdot (\xx)}{|\xx|^2}\right| \leq
|\nabla \vec u_0(\vec x)| \sigma (|\xx|).$$
We aim now at proving that $I^{133}_\d$ goes to $0$ as well. As the function
$$(\vec x,\vec x') \mapsto \frac{(\tilde{\vec
   u}(\vec x') - \tilde{\vec u}(\vec x) -\tilde{\vec P}(\vec x)  (\xx))\cdot
    (\xx)}{|\xx|^2}$$
is bounded, such convergence would follow as soon as we check that the functions
$$(\vec x,\vec x') \mapsto \r_\d(\xx) \, \frac{(\vec
   u_0(\vec x') - \vec u_0(\vec x) - \nabla
    \vec u_0(\vec x) (\xx))\cdot (\xx)}{|\xx|^2}  $$
converge to $0$ strongly in $L^1(\Omega \times \Omega)$. In case of a smooth
function $\vec v$ this would follow from the bound
$$\left|\r_\d(\xx) \frac{(\vec v(\vec x') - \vec v(\vec x) - \nabla
    \vec v(\vec x) (\xx))\cdot (\xx)}{|\xx|^2} \right| \leq C 
\r_\d(\xx) \| {\rm D}^2 \vec v||_\infty |\xx|$$
by arguing as for $I^{111}_\d$. Fix then $\varepsilon>0$ and choose $\vec
v\in C^{\infty} (\bar{\Omega};\Rn)$ such that $\vec w = \vec u_0 - \vec
v$ fulfills $\| \vec w
\|_{H^1(\Omega ;\Rn)} \leq \varepsilon$. One has that 
\begin{align*}
 & \int_\O \int_\O \left| \r_\d(\xx) \, \frac{(\vec
   u_0(\vec x') - \vec u_0(\vec x) - \nabla
    \vec u_0(\vec x) (\xx))\cdot (\xx)}{|\xx|^2} \right| \dd \vec x'
\, \dd \vec x \nonumber\\
&\leq \int_\O \int_\O \left| \r_\d(\xx) \, \frac{(\vec
   v(\vec x') - \vec v(\vec x) - \nabla
    \vec v(\vec x) (\xx))\cdot (\xx)}{|\xx|^2} \right| \dd \vec x'
\, \dd \vec x\nonumber\\
& + 
\int_\O \int_\O \left| \r_\d(\xx) \, \frac{(\vec
  w(\vec x') - \vec w(\vec x) - \nabla
    \vec w(\vec x) (\xx))\cdot (\xx)}{|\xx|^2} \right| \dd \vec x'
\, \dd \vec x
\end{align*}
The first term in the above right-hand side goes to $0$ as $\delta \to
0$ because $\vec
v$ is smooth and the second can be treated as follows:
\begin{align*}
 & \int_\O \int_\O \left| \r_\d(\xx) \, \frac{(\vec
  w(\vec x') - \vec w(\vec x) - \nabla
    \vec w(\vec x) (\xx))\cdot (\xx)}{|\xx|^2} \right|\, \dd \vec x'
\, \dd \vec x \nonumber\\
&\leq \int_\O \int_\O \r_\d(\xx) \left|  \nabla
    \vec w(\vec x) \right|\, \dd \vec x'
\, \dd \vec x +  \int_\O \int_\O \r_\d(\xx) \left|  \frac{(\vec
  w(\vec x') - \vec w(\vec x))  \cdot (\xx)}{|\xx|^2} \right|\, \dd \vec x'
\, \dd \vec x\\
&\leq n\int_\O  \left|  \nabla
    \vec w(\vec x) \right|\, 
\, \dd \vec x +  \int_\O \int_\O \r_\d(\xx)  |  \frac{|\vec
  w(\vec x') - \vec w(\vec x)|  }{|\xx|}  \, \dd \vec x'
\, \dd \vec x\\
&\leq c\| \vec w \|_{H^1(\Omega;\Rn)} \leq c \varepsilon ,
\end{align*}
where we have also used \cite[Th.\ 1]{BoBrMi01} (see also \cite[Eq.\ (5)]{Ponce04}). As $\varepsilon$
is arbitrary, we conclude
that $I^{131}_\d$ goes to $0$ as $\delta \to 0$.

All in all, we have proved that 
\begin{equation}
  \label{I1final}
    I^{1}_\d \to   n \int_\O \dashint_{\mathbb
    S^{n-1}}\big( \nabla \vec u_0(\vec x) - \vec P_0(\vec x)\big)\vec z\cdot \vec
  z   \, \big(\nabla \tilde{\vec u}(\vec x) - \tilde{\vec P} (\vec
   x)\big) \vec z \cdot \vec z \, \,\dd \mc{H}^{n-1}(\vec z)\, \dd
  \vec x\quad \text{as $\d \to 0$} .
\end{equation}

\bigskip

\noindent {\bf Step 4:  Integrals $I^2_\d$, $I^3_\d$, and $I^4_\d$.} One can discuss integral
$I^2_\d$ by following the analysis of integral $I^1_\d$. Indeed, the
two integrals correspond to each other  upon changing
$\mc{E}(\tilde{\vec u},\tilde{\vec P})$ there with $\div \tilde{\vec
  u}/n$ here.   In particular, we have that 
\begin{equation}
  \label{I2final}
I^2_\d \to  -  \int_\O \dashint_{\mathbb
  S^{n-1}} \left(\nabla \vec u_0( \vec x) - \vec P_0(\vec x) \vec z \cdot \vec z \right) \div \tilde{\vec
  u}(\vec x)\, \dd \mc{H}^{n-1}(\vec
z) \, \dd \vec x \quad \text{as $\d \to 0$.}
\end{equation}

As for $I^3_\d$, we decompose $I^3_\d = I^{31}_\d + I^{32}_\d$, where
\begin{align*}
  I^{31}_\d & = -\frac{1}{n} \int_\O\int_\O \r_\d(\vec x {-} \vec x') \,\div
\vec u_0 (\vec x) \, \frac{(\nabla \tilde{\vec
      u}(\vec x) - \tilde{\vec P}(\vec x)) (\xx)}{|\xx|^2} \cdot (\xx)
    \,
\dd \vec x' \, \dd \vec x,\\
I^{32}_\d&=-\frac{1}{n} \int_\O\int_\O \r_\d(\vec x {-} \vec x') \,\div
\vec u_0 (\vec x) \, \frac{(\tilde{\vec u}(\vec x') - \tilde{\vec u}(\vec x) - \nabla \tilde{\vec
      u}(\vec x) (\xx))}{|\xx|^2}\cdot (\xx) \,
\dd \vec x' \, \dd \vec x.
\end{align*}
As the integrand of $I^{31}_\d$ is positively homogeneous of degree
$0$ in $\xx$, one can use Lemma \ref{le:formulasB}.a in order to get
that 
$$ I^{31}_\d \to - \int_\O \dashint_{\mathbb
  S^{n-1}}\div
\vec u_0 (\vec x) \, (\nabla \tilde{\vec
      u}(\vec x) - \tilde{\vec P}(\vec x)) \vec z \cdot \vec z \,\,\dd \mc{H}^{n-1}(\vec
z) \, \dd \vec x \quad \text{as $\d \to 0$}.$$
As regards integral $I^{32}_\d$, one can simply reproduce the
argument of $I^{112}_\d$ in order to check that $I^{32}_\d\to 0$ as
$\d \to 0$. This allows us to conclude that 
\begin{equation}
  \label{I3final}
  I^3_\d \to - \int_\O \dashint_{\mathbb
  S^{n-1}}\div
\vec u_0 (\vec x) \, (\nabla \tilde{\vec
      u}(\vec x) - \tilde{\vec P}(\vec x)) \vec z \cdot \vec z \,\,\dd \mc{H}^{n-1}(\vec
z) \, \dd \vec x \quad \text{as $\d \to 0$}.
\end{equation}

The treatment of term
$I^4_\d$ is rather straightforward as
\begin{align}
  I^4_\d =  \frac{1}{n^2} \int_\O \div\vec u_0(\vec x) \, \div
  \tilde{\vec u}(\vec x) \left(\int_\O \r_\d(\vec x {-} \vec x')\, \dd
    \vec x'\right) \, \dd \vec x \to \frac{1}{n} \int_\O \div \vec u_0(\vec x) \, \div
  \tilde{\vec u}(\vec x)  \, \dd \vec x \quad \text{as $\d \to 0$.}\label{I4final}
\end{align}
where we have used that $\int_\O \r_\d(\vec x {-} \vec x')\, \dd
    \vec x' \to n$ as $\d \to 0$.

\bigskip

\noindent {\bf Conclusion of the proof.} By recollecting \eqref{lim1},
the decomposition \eqref{decomp},
and the limits \eqref{I1final}, \eqref{I2final}, \eqref{I3final}, and
\eqref{I4final} we conclude that 
\begin{align*}
 &\lim_{\d \to 0} A_\d\big((\vec u_\d,\vec P_\d),(\tilde{\vec u},\tilde{\vec
      P})\big) \stackrel{\eqref{lim1}}{=} \lim_{\d \to 0} \tilde A_\d\big((\vec u_\d,\vec P_\d),(\tilde{\vec u},\tilde{\vec
      P});\vec u_0\big) \stackrel{\eqref{decomp}}{=} \lim_{\d \to 0}\left(
      I^1_\d + I^2_\d + I^3_\d + I^4_\d\right)\nonumber  \\
&\stackrel{\eqref{I1final}}{=}    n \int_\O \dashint_{\mathbb
    S^{n-1}}\big( \nabla \vec u_0(\vec x) - \vec P_0(\vec x)\big)\vec z\cdot \vec
  z   \, \big(\nabla \tilde{\vec u}(\vec x) - \tilde{\vec P} (\vec
   x)\big) \vec z \cdot \vec z \, \dd \mc{H}^{n-1}(\vec z)\, \dd
  \vec x\\
&\stackrel{\eqref{I2final}}{-}  \int_\O \dashint_{\mathbb
  S^{n-1}} \left(\nabla \vec u_0(\vec x) - \vec P_0(\vec x) \vec z \cdot \vec z \right) \div \tilde{\vec
  u}(\vec x)\, \dd \mc{H}^{n-1}(\vec
z) \, \dd \vec x\\
&\stackrel{\eqref{I3final}}{-}  \int_\O \dashint_{\mathbb
  S^{n-1}}\div
\vec u_0 (\vec x) \, (\nabla \tilde{\vec
      u}(\vec x) - \tilde{\vec P}(\vec x)) \vec z \cdot \vec z \,\dd \mc{H}^{n-1}(\vec
z) \, \dd \vec x\\ 
&\stackrel{\eqref{I4final}}{+} \frac{1}{n}\int_\O \div \vec u_0(\vec x) \, \div
  \tilde{\vec u}(\vec x)  \, \dd \vec x\\
&= n \int_{\O} \dashint_{\Sn} \left( (\nabla \vec u_0 (\vec x) - \vec
  P_0 (\vec x)) \vec z \cdot \vec z - \frac{1}{n} \div \vec u_0 (\vec
  x) \right) \left( (\nabla \tilde{\vec u} (\vec x) - \tilde{\vec P}
  (\vec x)) \vec z \cdot \vec z - \frac{1}{n} \div \tilde{\vec u}
  (\vec x) \right) \, \dd \Hn (\vec z) \, \dd \vec x ,
\end{align*}
which proves the convergence of the $\alpha$ term of $B_\d$. This concludes the proof.
\end{proof}


\begin{thebibliography}{10}

\bibitem{BeMo14}
{\sc J.~C. Bellido and C.~Mora-Corral}, {\em Existence for nonlocal variational
  problems in peridynamics}, SIAM J. Math. Anal., 46 (2014),
~890--916.

\bibitem{Bellido15}
{\sc J.~C. Bellido, C. Mora-Corral,  and   P. Pedregal}, {\em Hyperelasticity as a
$\Gamma$-limit of Peridynamics when the horizon goes to zero},   
  Calc. Var. Partial Differential Equations, 54 (2015),
1643--1670.

\bibitem{Handbook}
{\sc F. Bobaru, J.~T. Foster, P.~H. Geubelle,  and   S.~A. Silling (Eds.)},
{\em Handbook of Peridynamic Modeling}, Advances in Applied
Mathematics, CRC Press, 2017.


\bibitem{BoBrMi01}
{\sc J.~Bourgain, H.~Brezis, and P.~Mironescu}, {\em Another look at {S}obolev
  spaces}, in Optimal Control and Partial Differential Equations, J.~L.
  Menaldi, E.~Rofman, and A.~Sulem, eds., IOS Press, 2001, ~439--455.

\bibitem{Brezis02}
{\sc H.~Brezis}, {\em How to recognize constant functions. {A} connection with
  {S}obolev spaces}, Uspekhi Mat. Nauk, 57 (2002), ~59--74.

\bibitem{DalMaso}
{\sc G.~{Dal Maso}},
 {\em An introduction to {$\Gamma$}-convergence},
\newblock Progress in Nonlinear Differential Equations and their Applications,
  8. Birkh\"auser Boston Inc., Boston, MA, 1993.

\bibitem{DeGiorgi} 
{\sc E.~De Giorgi and T.~Franzoni}, {\em Su un tipo di convergenza
  variazionale}, Atti Accad. Naz. Lincei Rend. Cl. Sci. Fis. Mat. Natur. (8),
  58 (1975), 842--850.

\bibitem{DuGuLeZh12}
{\sc Q.~Du, M.~Gunzburger, R.~B. Lehoucq, and K.~Zhou}, {\em Analysis and
  approximation of nonlocal diffusion problems with volume constraints}, SIAM
  Rev., 54 (2012), ~667Ð--696.

\bibitem{DuGuLeZh13ELAS}
\leavevmode\vrule height 2pt depth -1.6pt width 23pt, {\em Analysis of the
  volume-constrained peridynamic {N}avier equation of linear elasticity}, J.
  Elasticity, 113 (2013), ~193--217.

\bibitem{DuGuLeZh13}
\leavevmode\vrule height 2pt depth -1.6pt width 23pt, {\em A nonlocal vector
  calculus, nonlocal volume-constrained problems, and nonlocal balance laws},
  Math. Models Methods Appl. Sci., 23 (2013), ~493--540.

\bibitem{Emmrich}
{\sc E. Emmrich,  R.~B. Lehoucq,  and   D. Puhst}, {\em Peridynamics: a
  nonlocal continuum theory}, In: M. Griebel, M.~A. Schweitzer  (eds),
Meshfree Methods for Partial Differential Equations VI, Lecture Notes
in Computational Science and Engineering, vol 89, 45--65, Springer, Berlin,
Heidelberg, 2008.

\bibitem{Emmrich15}
{\sc E. Emmrich  and     D. Puhst}, {\em
Survey of existence results in nonlinear peridynamics in comparison
with local elastodynamics},  Comput. Methods Appl. Math., 15 (2015),
483--496. 

\bibitem{Emmrich16}
\leavevmode\vrule height 2pt depth -1.6pt width 23pt, {\em A short note on modeling damage in
peridynamics},  J. Elasticity, 123 (2016), 245--252.


\bibitem{GuLe10}
{\sc M.~Gunzburger and R.~B. Lehoucq}, {\em A nonlocal vector calculus with
  application to nonlocal boundary value problems}, Multiscale Model. Simul., 8
  (2010), ~1581--1598.

\bibitem{Han-Reddy}
{\sc W.~Han and B.~D. Reddy}, {\em Plasticity}, vol.~9 of Interdisciplinary
  Applied Mathematics, Springer, New York, second~ed., 2013.
\newblock Mathematical theory and numerical analysis.

\bibitem{Hu12}
{\sc W. Hu, Y.~D. Ha,  and   F. Bobaru}, {\em Peridynamic model for dynamic fracture in
unidirectional fiber-reinforced composites},
  Comp. Methods Appl. Mech. Engrg., 217--220 (2012), 247--261.

\bibitem{Madenci16}
{\sc E. Madenci  and    S. Oterkus}, {\em Ordinary state-based peridynamics for plastic
deformation according to von Mises yield criteria with isotropic
hardening},  J. Mech. Phys. Solids, 86 (2016), 192--219. 

\bibitem{Mengesha12}
{\sc T.~Mengesha}, {\em Nonlocal {K}orn-type characterization of {S}obolev
  vector fields}, Commun. Contemp. Math., 14 (2012), ~1250028, 28.

\bibitem{MeDu15}
{\sc T.~Mengesha and Q.~Du}, {\em On the variational limit of a class of
  nonlocal functionals related to peridynamics}, Nonlinearity, 28 (2015),
  ~3999--4035.

\bibitem{MeDu14}
\leavevmode\vrule height 2pt depth -1.6pt width 23pt, {\em Nonlocal
  constrained value problems for a linear Peridynamic Navier
  equation}, J. Elasticity, 116 (2014), 27--51.

\bibitem{MeSp15}
{\sc T.~Mengesha and D.~Spector}, {\em Localization of nonlocal gradients in
  various topologies}, Calc. Var. Partial Differential Equations, 52 (2015),
  ~253--279.

\bibitem{MielkeRoubicek15}
{\sc A.~Mielke and T.~Roub\'\i\v{c}ek}, {\em Rate-independent
  systems. Theory and application},
  vol.~193 of Applied Mathematical Sciences, Springer, New York, 2015.
\newblock 

\bibitem{MRS}
{\sc A.~Mielke, T.~Roub\'\i\v{c}ek, and U.~Stefanelli}, {\em {$\Gamma$}-limits
  and relaxations for rate-independent evolutionary problems}, Calc. Var.
  Partial Differential Equations, 31 (2008), ~387--416.

\bibitem{Ponce04}
{\sc A.~C. Ponce}, {\em An estimate in the spirit of {P}oincar\'e's
  inequality}, J. Eur. Math. Soc.  (JEMS),  6 (2004), ~1--15.

\bibitem{Ponce04b}
\leavevmode\vrule height 2pt depth -1.6pt width 23pt, {\em A new approach to
  {S}obolev spaces and connections to {$\Gamma$}-convergence}, Calc. Var.
  Partial Differential Equations, 19 (2004), ~229--255.

\bibitem{Seleson09}
{\sc P. Seleson, M.~L. Parks, M. Gunzburger,  and   R.~B. Lehoucq}, {\em Peridynamics as
an upscaling of molecular dynamics},  Multiscale Model. Simul., 8  (2009), 204--227.


\bibitem{Silling00}
{\sc S.~A. Silling}, {\em Reformulation of elasticity theory for
  discontinuities and long-range forces}, J. Mech. Phys. Solids, 48 (2000),
  ~175--209.

\bibitem{Silling10}
{\sc S.~A. Silling}, {\em Linearized theory of peridynamic states},
J. Elasticity, 99 (2010), 85--111.

\bibitem{Silling-Lehoucq08}
{\sc S.~A. Silling and R.~B. Lehoucq}, {\em Convergence of Peridynamics
to classical elasticity theory},
J. Elasticity, 93 (2008), 13--37.

\bibitem{survey}
\leavevmode\vrule height 2pt depth -1.6pt width 23pt, {\em Peridynamic
  theory of solid mechanics}, Adv. Appl. Mech., 44 (2010), 73--166.

\end{thebibliography}
\end{document}